\documentclass{amsart}[12pt,article]

\usepackage{amsmath,amssymb,amscd,amsthm,indentfirst}
\usepackage[all]{xy}
\usepackage{enumerate}
\usepackage{hyperref}
\usepackage{graphicx}
\usepackage{multirow}
\usepackage{pgf,tikz,pgfplots}
\pgfplotsset{compat=1.15}
\usepackage{mathrsfs}
\usetikzlibrary{arrows}
\usepackage{wrapfig}
\usepackage{cutwin}
\usepackage{lipsum}
\usepackage{color}
\usepackage{accents}
\usepackage{yhmath}
\usepackage{import}
\usepackage{pdfpages}

\newtheorem{introtheorem}{Theorem}

\newtheorem*{introtheorem*}{Theorem}
\newtheorem{thm}{Theorem}[section]
\newtheorem{prop}[thm]{Proposition}

\newtheorem{lem}[thm]{Lemma}

\theoremstyle{definition}
\newtheorem{defn}[thm]{Definition}
\newtheorem{remark}[thm]{Remark}

\newtheorem{exa}[thm]{Example}

\newenvironment{dedication}
   {\vspace{0ex}\begin{quotation}\begin{center}\begin{em}}
   {\par\end{em}\end{center}\end{quotation}}

\def\A{\ensuremath{\mathcal{A}}}

\def\SB{\ensuremath{\mathcal{S}}}
\def\C{\ensuremath{\mathcal{C}}}
\def\D{\ensuremath{\mathcal{D}}}
\def\Q{\ensuremath{\mathcal{Q}}}

\def\FF{\ensuremath{\mathbb{F}}}

\def\ZZ{\ensuremath{\mathbb{Z}}}

\def\QQ{\ensuremath{\mathbb{Q}}}

\def\ii{\mathbf i}
\def\jj{\mathbf j}

\def\rk{\operatorname{m}}

\def\Out{\mathrm{Out}}

\def\Aut{\mathrm{Aut}}

\def\sgn{\operatorname{sgn}}

\def\diag{\operatorname{diag}}

\def\tuples{\mathcal{T}}

\def\ord{\operatorname{ord}}

\def\barsub{\mathbf{C}}

\newcommand{\GL}{\operatorname{GL}\nolimits}
\newcommand{\SL}{\operatorname{SL}\nolimits}
\newcommand{\PSL}{\operatorname{PSL}\nolimits}
\newcommand{\GU}{\operatorname{GU}\nolimits}
\newcommand{\SU}{\operatorname{SU}\nolimits}
\newcommand{\PSU}{\operatorname{PSU}\nolimits}
\newcommand{\PGU}{\operatorname{PGU}\nolimits}

\newcommand{\transpose}{{\operatorname{T}\nolimits}}

\newcommand{\ls}[2]{{}^{{#1}\!}{#2}}

\newcommand{\qbox}[1]{\quad\hbox{#1}\quad}

\date{\today}
\title{Quillen's conjecture and unitary groups}

\author{Antonio D\'{i}az Ramos}
\address{Departamento de {\'A}lgebra, Geometr{\'\i}a x Topolog{\'\i}a,
Universidad de M{\'a}\-la\-ga, Apdo correos 59, 29080 M{\'a}laga,
Spain.}
\email{adiazramos@uma.es}

\begin{document}

\subjclass[2010]{55U10; 05E45, 20J99.}

\maketitle

\begin{dedication}
To Mariola and María.
\end{dedication}

\begin{abstract}
We prove that the Quillen posets $\A_p(H)$ of $p$-extensions $H$ of simple unitary groups have non-zero homology in the largest possible dimension, with just a few exceptions. This establishes a conjecture raised by Aschbacher–Smith in 1992. In particular, by their work and a more recent article by Piterman–Smith, Quillen’s conjecture on the $p$-subgroup posets holds for odd primes.
\end{abstract}

\section{Introduction}
\label{section:introduction}

In \cite[Corollary 2]{Brown75}, K.S. Brown proved the following result: If $\SB_p(G)$ is the poset of the non-trivial $p$-subgroups of a finite group $G$ ordered by inclusion, and $|\SB_p(G)|$ denotes its topological realization, then $\chi(|\SB_p(G)|)\equiv 1$ modulo the highest power of $p$ dividing the order of $G$, where $\chi$ stands for Euler characteristic. Thus, here, the poset $\SB_p(G)$ relates a topological invariant to an algebraic invariant and, since then, $p$-subgroups posets have linked algebraic topology and finite group theory from other perspectives,  including homotopy equivalences and Tits buildings \cite{Quillen1978}, group cohomology and geometries \cite{Webb1987}, homology decompositions \cite{BS2008,Dwyer1997}, and, more recently, fusion systems \cite{AKO2011}.

In \cite{Quillen1978}, Quillen focused on the poset $\A_p(G)$, which consists of the non-trivial elementary abelian $p$-subgroups of the finite group $G$ ordered by inclusion, proving that $|\SB_p(G)|$ and $|\A_p(G)|$ are homotopy equivalent, and studying the relation between the algebraic properties of $G$ and the topological properties of these spaces. Among his results, it stands out ``conical contractibility'' \cite[Proposition 2.4]{Quillen1978}: If we denote by $O_p(G)$ the largest normal $p$-subgroup of $G$, then the condition $O_p(G)\neq 1$ implies that $|\A_p(G)|$ is contractible. Moreover, Quillen conjectured that the converse statement should also hold, \cite[Conjecture 2.9]{Quillen1978},
\[
\text{Quillen's conjecture: If $O_p(G)=1$, then $|\A_p(G)|$ is not contractible.}
\]
This conjecture was established by Quillen himself for groups of $p$-rank at most 2, groups of Lie type in characteristic $p$, and solvable groups. The former case was recently extended to $p$-rank $3$ \cite{PSC2021}, and the latter case was extended to $p$-solvable groups by Alperin and others  \cite[Theorem 8.2.12]{Smith2011}, employing the Classification of the Finite Simple Groups (CFGS for short), and by Hawkes and Isaacs without the CFSG if the Sylow $p$-subgroup is abelian \cite{HawksIsaacs1988}. 

We refer the reader to \cite[Chapter 8]{Smith2011} for a historical account and we discuss next the contributions \cite{AK1990,AS1993,P2021,PS2022}: In \cite{P2021}, it is proven that Quillen's conjecture is equivalent to the version obtained by replacing the non-contractibility conclusion with non-triviality of reduced integral homology. Nevertheless, in general, these works establish cases of the stronger version involving non-triviality of reduced rational homology. In doing so, a non-zero homology cycle is proven to exist but, in most cases, it is not explicitly determined. 

The greatest advances towards proving the conjecture in full generality are \cite{AS1993} by Aschbacher and Smith and \cite{PS2022} by Piterman and Smith. These works are based on the analysis of the components of a hypothetical minimal counterexample to the conjecture, and they employ the CFSG. Very roughly, the method in \cite{AS1993} states that, under certain group-theoretic reductions on a group $G$ and for $p>5$, the following dichotomy should hold: for any component $L$ of $G$, either every $p$-extension of $L$ satisfies the Quillen dimension property for $p$, or $L$ possesses a so-called Robinson subgroup (see \cite[p.491]{AS1993} and \cite{Robinson88}). 

In order to explain the notions employed in the previous paragraph, let $\rk_p(G)$ denote the $p$-rank of the finite group $G$, and recall that this number is the largest integer $m$ such that $G$ has an elementary abelian $p$-subgroup of order $p^m$. Then we say that $G$ has the \emph{Quillen-dimension property} for $p$, written as $\Q\D_p$, if $|\A_p(G)|$ has non-zero reduced rational homology in its maximal dimension, i.e., in degree $\rk_p(G)-1$, see Definition \ref{def:QDp}. This property was implicitly employed by Quillen to establish his conjecture in the aforementioned cases, and plays a central role in the works of Aschbacher and Smith and of Piterman and Smith. On the other hand, the $p$-extensions of a group $L$ are the split extensions $LB$ of $L$ by an elementary abelian $p$-group $B$ of outer automorphisms of $L$. The case $B=1$ is allowed, so that $L$ is a $p$-extension of itself. The components of a group are defined in Section \ref{section:group_theoretical_preliminaries}.

We remark that the primes $p=2,3,5$ were excluded from the analysis in \cite{AS1993} since most of the preliminary results needed to prove the main theorem of \cite{AS1993} rely on the CFSG and facts about $p$-extensions when $p$ is odd. In \cite{PS2022}, the authors refine the dichotomy proposed above under similar restrictions and extend it to all odd primes, and carry out relevant progress for $p=2$; see more comments on this latter case in page \pageref{intro:p=2_detail}. To be specific, we state Aschbacher and Smith's main result below, and we note that Piterman and Smith's improvement on this result in \cite[Theorem 1.1]{PS2022} implies the same conclusion while replacing the hypothesis (i) by the weaker condition of $p$ being and odd prime. 

\begin{introtheorem*}[{\cite[Main Theorem]{AS1993}}]
Let $G$ be a finite group. Assume that
\begin{enumerate}
\item[(i)] $p$ is a prime with $p > 5$; and
\item[(ii)] whenever $G$ has a unitary component $\PSU_n(q)$ with $p$ dividing $q+1$ and $q$ odd, then $\Q\D_p$ holds for all $p$-extensions of $\PSU_l(q^{p^e})$ with $l\leq n$ and $e\in \ZZ$.
\end{enumerate}
Then $G$ satisfies Quillen's conjecture for $p$.
\end{introtheorem*}

Although unitary groups $\PSU_n(q)$ satisfy Quillen’s conjecture by the work of Aschbacher and Kleidman in \cite{AK1990}, it has been unknown so far whether $\Q\D_p$ holds for them and their $p$-extensions when $p$ is an odd prime dividing $q+1$. Therefore, by \cite{AS1993,PS2022}, showing that $p$-extensions of unitary groups $\PSU_n(q)$, with $p,q$ odd and $p\mid q+1$, satisfy $\Q\D_p$ is the only obstruction left to conclude with the proof of Quillen's conjecture for odd primes. 

In \cite{Diaz2016} and \cite{DiazMazza2020}, a new geometric method is introduced and employed to prove the Quillen dimension property for $p$ in known cases, namely, for $p$-solvable groups if $p$ is any prime, and for groups of type alternating, symmetric, linear, unitary, symplectic, and orthogonal for some odd primes $p$. This approach works but for a few exceptions, including the case $\PSU_n(q)$ when $p$ is odd and $p\mid q+1$. In the current work, we elaborate on this geometric method and show that unitary groups and their $p$-extensions do satisfy the Quillen dimension property, as conjectured by Aschbacher and Smith in \cite[Conjecture 4.1]{AS1993}, see Theorem \ref{thm:mainPSU} below. We proceed via a seemingly indirect path, as we deal first with general unitary groups.

\begin{introtheorem}\label{thm:mainPGU}
Let $q$ be a prime power, $p$ an odd prime dividing $q+1$, $n\geq 2$, and $G$ equal to $\PGU_n(q)$ or $\PGU_n(q)$ extended by field automorphisms of order $p$, where we exclude $(p,q)=(3,2)$ in the former case, and $n=2$ and $(p,q)=(3,8)$ in the latter case. Then $\widetilde{H}_{\rk_p(G)-1}(|\A_p(G)|;\ZZ)\neq 0$ and hence $G$ has the Quillen dimension property for $p$.
\end{introtheorem}

We emphasize that, in this result, as well as in \cite{Diaz2016} and \cite{DiazMazza2020}, explicit non-zero homology cycles are constructed, as opposed to what has been done in previous works. The $p$-extensions of $\PSU_n(q)$ are listed in Proposition \ref{prop:pextensions_PSUnq}, and that   Theorem A suffices to prove the Quillen dimension property for all these $p$-extensions, follows by an argument that traces back to \cite{AS1993} and that preserves the explicit constructions. We explicitly state this result below, and we include a detailed proof in Section \ref{section:TheoremB}.
 
\begin{introtheorem}\label{thm:mainPSU}
Let $q$ be a prime power, $p$ an odd prime dividing $q+1$, $n\geq 2$, and $G=\PSU_n(q)B$ a $p$-extension,  where we exclude $(p,q)=(3,2)$, and $(p,q)=(3,8)$ if $B$ contains field automorphisms of order $p$. Then $\widetilde{H}_{\rk_p(G)-1}(|\A_p(G)|;\ZZ)\neq 0$ and hence $G$ has the Quillen dimension property for $p$.
\end{introtheorem} 

This theorem shows that \cite[Conjecture 4.1]{AS1993} is valid. Moreover, by the result \cite[Main Theorem]{AS1993} by Aschbacher and Smith and its extension to all odd primes in \cite[Theorem 1.1]{PS2022} by Piterman and Smith, we conclude the following result.

\begin{introtheorem}\label{thm:Quillen}
Let $G$ be a finite group and $p$ be an odd prime such that $O_p(G)=1$. Then $\widetilde{H}_*(|\A_p(G)|;\QQ)\neq 0$. In particular, Quillen's conjecture holds for odd primes.
\end{introtheorem}

There are several reasons that explain why these techniques do not cover yet the case $p=2$ \label{intro:p=2_detail}. First, Theorems \ref{thm:mainPGU} and \ref{thm:mainPSU} and most of the aforementioned works that establish the Quillen dimension property for odd $p$ in various cases do not immediately extend to $p=2$. Some obstructions to handle are that the description of $p$-ranks,  maximal elementary abelian $p$-subgroups, and outer automorphism groups, are more involved for $p=2$ than for odd $p$, see \cite[Section 2]{AS1993}. Secondly, a dichotomy similar to the one introduced above for odd $p$ is not yet completed for $p=2$, in particular, a so-called $\Q\D$-list of simple groups possibly failing the Quillen dimension property, cf. \cite[Theorem 3.1]{AS1993}, is still lacking. As mentioned above, relevant progress in this direction has been made by Piterman and Smith in \cite{PS2022}, and other elimination results, besides the Quillen dimension property, are needed to prepare this missing list.

\textbf{Preview of the construction of the non-trivial homology class of Theorem \ref{thm:mainPGU}:} Given an elementary abelian subgroup $E$ of rank $\rk_p(G)$, a simplicial chain $\barsub_E$ is constructed that corresponds to the barycentric subdivision of an $(\rk_p(G)-1)$-simplex with $E$ as barycenter, see Definition \ref{def:ZEa} and Remark \ref{rmk:barycentric_subdivision} for an illustration. The goal is then to find a subset $X$ of $G$ such that a linear combination of the $X$-conjugates of $\barsub_E$ is a cycle, see Definition \ref{def:ZGXA} and Theorem \ref{thm:noncontractiblemoregeneral}. For the case of $\PGU_n(q)$, we choose $E$ to be spanned by diagonal elements, and $X$ to contain permutation matrices as well as a quasi-reflection; see Subsection \ref{subsection:quasi-reflections} for the definition of quasi-reflections and Remark \ref{rmk:why_x_and_Y-+} for further details on these choices.  Geometrically, the resulting complex  is a triangulation of the sphere $S^{n-2}$, similar to the Coxeter complex of the symmetric group on $n$ letters and it has $n!$ $(n-2)$-simplices. This contrast with much of the previous literature, e.g., \cite[Theorem 11.2(b)]{Quillen1978}, \cite{Diaz2016}, \cite{DiazMazza2020} where the sphere $S^{n-2}$ is the standard $(n-1)$-folded join of $S^0$ and has $2^{n-1}$ $(n-2)$-simplices. See Figure \ref{fig:triangulationPSLPGUr=3} below. 

For the case of $\PGU_n(q)$ extended by field automorphisms of order $p$, we consider the triangulation of $S^{n-2}$ obtained for $\PGU_n(q^{1/p})$ and adjoin the field automorphism to $E$, obtaining in this way a triangulation of an $(n-1)$-dimensional disk. Then we conjugate this triangulation by a diagonal element that fixes its boundary but does not centralize the field automorphism. The resulting complex is a suspension or double cone on the triangulation of $S^{n-2}$ constructed for $\PGU_n(q^{1/p})$, with the field automorphism in an apex and a conjugate of it in the opposite apex. See Example \ref{example:n2_quasireflections} and the figures there.

\begin{figure}[h!]
\centering
\includegraphics{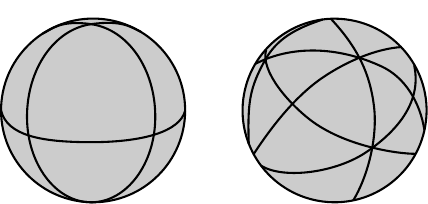}
\caption{Triangulations of the sphere $S^2$ constructed for $\PSL_4(q)$ in \cite{DiazMazza2020} (left), and for $\PGU_4(q)$ in Theorem \ref{thm:mainPGU} (right).}
\label{fig:triangulationPSLPGUr=3}
\end{figure}

\textbf{Outline of the paper:} We include group theoretical preliminaries in Section \ref{section:group_theoretical_preliminaries} and topological preliminaries in Section \ref{section:topological_preliminaries}. Theorem \ref{thm:mainPGU} is consequence of Theorems \ref{thm:QDp_for_PGUn} and \ref{thm:QDp_for_PGUnfield}, while the proof of Theorem \ref{thm:mainPSU} is to be found in Section \ref{section:TheoremB}.

\textbf{Acknowledgements:} I am grateful to the anonymous referee for the careful reading and for the many helpful comments. I am in debt to Stephen Smith and Kevin Piterman for their detailed feedback on several previous versions of this work.



\section{Group theoretical preliminaries}\label{section:group_theoretical_preliminaries}

We denote by $G$ a finite group and we fix some group-theoretical notation: If $g,h$ belong to $G$, we write $\ord(g)$ for the order of $g$, we denote conjugation by $g$ as $c_g(h)=\ls{g}h=ghg^{-1}$, and we define the commutator of $g$ and $h$ as $[g,h]=g^{-1}h^{-1}gh$. If $H\leq G$ is a subgroup of $G$, we denote its centralizer in $G$, its normalizer in $G$, its centre, and its commutator subgroup, as $C_G(H)$, $N_G(H)$, $Z(H)$, and $[H,H]$, respectively. The Classification of the Finite Simple Groups states that every finite simple group is isomorphic to one of the following, see \cite{GLSI},
\begin{enumerate}
\item[(i)] a cyclic group of prime order,
\item[(ii)] an alternating group of degree $n\geq 5$,
\item[(iii)] a finite simple group of Lie type,
\item[(iv)] one of the twenty-six sporadic groups.
\end{enumerate}
Only the groups in (i) are abelian, so that the groups in (ii,iii,iv) are the non-abelian simple groups. The projective special unitary group $\PSU_n(q)$, which is defined below and is the theme of this paper, is a finite group of Lie type, and it is simple for $n\geq 2$ except for a few exceptions for $n=2$, see below.

The components of a finite group are defined as its quasi-simple subnormal subgroups, and we recall that a finite group $H$ is quasi-simple if it satisfies that $H=[H,H]$ and that $H/Z(H)$ is simple. The special unitary group $\SU_n(q)$ is, in general, quasi-simple, see below. As discussed in the introduction, the analysis of the components of given finite group is a  cornerstone for the results of Aschbacher and Smith \cite{AS1993} and of Piterman and Smith \cite{PS2022}.

\subsection{General and special unitary groups.}\label{subsection:preliminaries_unitary_groups} Throughout the paper we assume $n\geq 2$, as for $n=1$ we would get abelian groups in the forthcoming constructions. Let $\GL_n(q)$ denote the general linear group over the field $\FF_q$ of $q=s^l$ elements, $s$ a prime, and let $\SL_n(q)$ denote the special linear group, which is the normal subgroup of $\GL_n(q)$ consisting of the matrices of determinant $1$. 

We denote by $I_n$ the $n\times n$ identity matrix, and recall that the field $\FF_{q^2}$ of $q^2=s^{2l}$ elements has the involution automorphism
\[
\alpha\mapsto \overline{\alpha}=\alpha^q,
\]
and that $\FF_q=\{\alpha\in \FF_{q^2}|\alpha=\overline{\alpha}\}$. Over an $n$-dimensional vector space over $\FF_{q^2}$, we may define the following Hermitian form employing the involution above, 
\[
f((\alpha_1,\ldots,\alpha_n),(\beta_1,\ldots,\beta_n))=\alpha_1\overline{\beta_1}+\cdots+\alpha_n\overline{\beta_n}.
\]
We define the general unitary group $\GU_n(q)$ as the subgroup of unitary matrices of $\GL_n(q^2)$ with respect to the form $f$,  i.e.,
\[
\GU_n(q)=\{x\in \GL_n(q^2) ~|~ x\overline{x}^\transpose=I_n\},
\]
where $\overline{x}$ is obtained from $x$ by applying $\overline{\cdot}$ to each entry and $^\transpose$ denotes transpose. Analogously, we define the special unitary group $\SU_n(q)$ as the subgroup of unitary matrices of $\SL_n(q^2)$.

We define $T$ as the subgroup of diagonal matrices of $\GU_n(q)$, i.e., elements of the form $\diag(\alpha_1,\ldots,\alpha_n)$ with $\alpha_i\overline{\alpha_i}=\alpha_i^{q+1}=1$. The centers of $\GU_n(q)$ and $\SU_n(q)$ consist of the diagonal matrices $\diag(\alpha,\ldots,\alpha)$ with $\alpha\overline{\alpha}=1$ and $\alpha\overline{\alpha}=\alpha^n=1$, respectively, giving the following descriptions as cyclic groups,
\[
Z(\GU_n(q))=C_{q+1} \text{ and }Z(\SU_n(q))=C_{(n,q+1)},
\]
where $(n,q+1)$ denotes greatest common divisor. The quotient groups 
\[
\PGU_n(q)=\GU_n(q)/Z(\GU_n(q))\text{ and }\PSU_n(q)=\SU_n(q)/Z(\SU_n(q))
\]
are the projective general unitary group and the projective special linear group, respectively. Moreover, $\SU_n(q)$ is perfect with simple quotient $\PSU_n(q)$, 
\[
Z(\SU_n(q))\to \SU_n(q)\to \PSU_n(q)=\SU_n(q)/Z(\SU_n(q)),
\]
i.e., it is quasi-simple, but for the cases $\SU_2(2)$, $\SU_2(3)$, and $\SU_3(2)$, see \cite[Chapter 11]{Grove}. In fact, the groups 
\begin{equation}\label{equ:small_group_exceptions}
\text{$\PSU_2(2)$, $\PSU_2(3)$, $\PSU_3(2)$}
\end{equation}
have $O_p(G)\neq 1$ for $p=3,2,3$ respectively (note this is the only prime dividing $q+1$), and the same happens to the groups $\PGU_2(2)$, $\PGU_2(3)$, and $\PGU_3(2)$, and to the extensions of $\PSU_3(8)$ and of $\PGU_3(8)$ by field automorphisms of order $3$ with $p=3$.  Thus, by ``conical contractibility'' \cite[Proposition 2.4]{Quillen1978}, $\Q\D_p$ cannot hold for these groups, explaining the cases excluded from Theorems \ref{thm:mainPGU} and \ref{thm:mainPSU}. To finish this subsection, we recall that
\begin{equation}\label{equ:PSU-PGU-cyclic}
\PGU_n(q)/\PSU_n(q)=C_{(n,q+1)}.
\end{equation}

\subsection{Automorphisms of unitary groups.} \label{subsection:automorphisms_unitary:groups} We define below, following \cite[Definition 2.5.13]{GLSIII}, the \emph{diagonal}, \emph{inner-diagonal}, and \emph{field automorphisms} of $\PSU_n(q)$, and we will employ the same names for the corresponding induced automorphisms in $\PGU_n(q)$.

Thus, we declare the \emph{diagonal} automorphisms of $\PSU_n(q)$ (and of $\PGU_n(q)$) as those induced by conjugation by elements from $\PGU_n(q)\setminus \PSU_n(q)$, and the \emph{inner-diagonal} automorphisms of $\PSU_n(q)$ (and of $\PGU_n(q)$) as those induced by conjugation by elements from $\PGU_n(q)$. In addition, the Frobenius map induces the automorphism of $\SU_n(q)$ (and of $\GU_n(q)$) that acts as $\lambda\mapsto \lambda^s$ on the matrix entries, and descends to an automorphism $\phi$ of $\PSU_n(q)$ (and of $\PGU_n(q)$). We define, see \cite[Definition 2.5.13(c1)]{GLSIII}, the \emph{field automorphisms} of $\PSU_n(q)$ (and of $\PGU_n(q)$) as the $\PGU_n(q)$-conjugates of a power of $\phi$ of odd order. In particular, with these definitions, all cyclic subgroups of field automorphisms of $\PGU_n(q)$ of the same odd order are $\PGU_n(q)$-conjugated. We are interested in the case that there exists a field automorphism of odd order $p$.

\begin{defn}\label{def:Phi}
If $q=s^{pl}$ for $p$ an odd prime, we set $\Phi=\phi^{2l}$, which is an order $p$ field automorphism of $\PSU_n(q)$ (and of $\PGU_n(q)$).
\end{defn}

We finish this subsection by recalling the $p$-ranks and $p$-extensions of $\PSU_n(q)$ as well as the $p$-ranks of $\PGU_n(q)$ and $\PGU_n(q)\langle \Phi\rangle$. The next result will be employed in Section \ref{section:TheoremB}, and note that the hypotheses below that $p$ is odd and that $(p,q)\neq (3,2)$ exclude the non-simple groups in Equation \eqref{equ:small_group_exceptions}. We recall that, for an integer $l$, $(l)_p$ denotes the largest power of $p$ dividing $l$.

\begin{prop}\label{prop:pextensions_PSUnq}
Assume that $n\geq 2$, $p$ is odd, $p\mid(q+1)$, and $(p,q)\neq (3,2)$. Then the $p$-ranks of $\PSU_n(q)$ and $\PGU_n(q)$ are as follows,
\begin{align*}
\rk_p(\PSU_n(q))&=\begin{cases}
n-2&\text{ if $(n)_p\geq(q+1)_p$ and $n>3$,}\\
n-1&\text{ otherwise.}\\
\end{cases}\\
\rk_p(\PGU_n(q))&=n-1.
\end{align*}
In addition, $\rk_p(\PGU_n(q)\langle \Phi\rangle)=n$ and, for $n\geq 3$, the $p$-extensions of $L=\PSU_n(q)$ are of the form $LB$ and have the $p$-ranks indicated below, with
\begin{enumerate}
\item[(i)]  $B=1$, and hence $LB=L$ and $\rk_p(LB)=\rk_p(\PSU_n(q))$, or
\item[(ii)] $B$ is cyclic generated by a field automorphism of order $p$ and $\rk_p(LB)=\rk_p(\PSU_n(q))+1$, or
\item[(iii)] $B$ is cyclic generated by a diagonal automorphisms of order $p$ and $\rk_p(LB)=n-1$, or
\item[(iv)]  $B$ has $p$-rank $2$, it is generated by a field and a diagonal automorphism of order $p$, and $\rk_p(LB)=n$.
\end{enumerate}
\end{prop}
\begin{proof}
The values for the $p$-ranks in the display of the statement are obtained from \cite[I.10-6(1)]{GL} and  \cite[I.10-2(2)]{GL}, see also \cite[Section 4]{AS1993}. For the $p$-extensions, recall that, for $n\geq 3$, we have, by  \cite[Theorem 2.5.12]{GLSIII}, that
\[
\Aut(\PSU_n(q))=\langle \PGU_n(q),\phi\rangle,
\]
but for $\PSU_3(2)$, which is included in the non-simple exceptions of Equation \eqref{equ:small_group_exceptions}. Quotienting by the inner automorphism group, we get the outer automorphism group,
\[
\Out(\PSU_n(q))=C_{(n,q+1)}\rtimes C_{2l},
\]
where $C_{(n,q+1)}$ is isomorphic the quotient group in Equation \eqref{equ:PSU-PGU-cyclic}, $C_{2l}$ is generated by $\phi$, and the action for this semi-direct product is described in \cite[Theorem 2.5.12(g)]{GLSIII}. This gives rise to the types of $p$-extensions in points (i-iv) of the statement by \cite[Proposition 4.9.1(d)]{GLSIII}. Finally, for the $p$-ranks of the $p$-extensions, we use that $\PSU_n(q)\subseteq \PGU_n(q)$ and \cite[Section 4]{AS1993}.
\end{proof}

\subsection{Quasi-reflections in $\GU_n(q)$.}\label{subsection:quasi-reflections}

For $f$ the Hermitian form introduced in Subsection \ref{subsection:preliminaries_unitary_groups}, we say that a vector $v$ is \emph{isotropic} if $f(v,v)=0$ and that it is \emph{anisotropic} otherwise.  Then, an anisotropic vector $v$ and an element $\alpha\in \FF^*_{q^2}$ with $\alpha\overline{\alpha}=1$, determine an element of $\GU_n(q)$, called \emph{quasi-reflection}, as follows, see \cite[p.94]{Grove},
\[
X_{v,\alpha}(u)=u+(\alpha-1)\frac{f(u,v)}{f(v,v)}v.
\]
In the next result, we construct a quasi-reflection that will be used in Section \ref{section:PGU}.

\begin{prop}\label{prop:a_quasi_reflection}
If $n\geq 2$ and $q\geq 3$, then there exists a quasi-reflection $x\in \GU_n(q)$, such that  $x$ is not a monomial matrix and that, for any  diagonal element $t=\diag(\alpha_1,\ldots,\alpha_n)\in T$, $\ls{x}{t}$ is non-diagonal if $\alpha_1\neq \alpha_2$ and $\ls{x}t=t$ if $\alpha_1=\alpha_2$.
\end{prop}
\begin{proof}
We will define the element $x$ as a quasi-reflection of the form 
\[
x=X_{v,\gamma},
\]
where $v=(1,\beta,0,\ldots,0)^\transpose$ and the scalars $\beta,\gamma\in \FF^*_{q^2}$ are chosen below subject to
\begin{equation}\label{equ:QR-beta_gamma_basic}
\gamma\overline{\gamma}=1\text{ and }1+\beta\overline{\beta}\neq 0.
\end{equation}
It is straightforward that $x$ defined in this way is the block diagonal matrix 
\begin{equation}\label{equ:x_in_2x2_square}
x=\diag(X_{(1,\beta)^\transpose,\gamma},I_{n-2}).
\end{equation}
In particular, it is enough to solve the case $n=2$, and we do so below.

For arbitrary values $\beta,\gamma\in \FF^*_{q^2}$ satisfying Equation \eqref{equ:QR-beta_gamma_basic}, a straightforward computation shows that
\[
X_{(1,\beta)^\transpose,\gamma}=\frac{1}{\zeta}\begin{pmatrix}
\gamma+\beta\overline{\beta} & (\gamma-1)\overline{\beta} \\ (\gamma-1)\beta & 1+\gamma\beta\overline{\beta}\end{pmatrix},
\]
where $\zeta=1+\beta\overline{\beta}$, as well as that $\ls{\left(X_{(1,\beta)^\transpose,\gamma}\right)}\diag(\alpha_1,\alpha_2)$ equals
\[
\frac{1}{\zeta^2}\begin{pmatrix}
\alpha_1\zeta^2+(\alpha_2-\alpha_1)(\gamma-1)(\overline{\gamma}-1)\beta\overline{\beta} & (\alpha_1-\alpha_2)(\overline{\gamma}-1)\overline{\beta}(\gamma+\beta\overline{\beta})\\
(\alpha_1-\alpha_2)(\gamma-1)\beta(\overline{\gamma}+\beta\overline{\beta}) & \alpha_2\zeta^2+(\alpha_1-\alpha_2)(\gamma-1)(\overline{\gamma}-1)\beta\overline{\beta}
\end{pmatrix}.
\]
Thus, for the conclusions of the proposition to hold, we need to choose $\beta,\gamma\in \FF_{q^2}$ satisfying Equation \eqref{equ:QR-beta_gamma_basic} together with
\begin{equation}\label{equ:QR-beta_gamma_advanced}
\gamma\neq 1\text{, }\beta\neq 0\text{, and }\gamma+\beta\overline{\beta}\neq 0.
\end{equation}
We proceed as follows: First, recall that the norm map, $\FF^*_{q^2}\to \FF^*_q$, $\alpha\mapsto \alpha\overline{\alpha}$, is surjective. Then, as $q\geq 3$ by hypothesis, there exist $q-1\geq 2$ elements in the co-domain of the norm map and we can choose $\beta\in \FF^*_{q^2}$ with $\beta\overline{\beta}\neq -1$. In particular, $\beta\neq 0$, $1+\beta\overline{\beta}\neq 0$, and we are left with choosing $\gamma\in \FF_{q^2}$ such that
\[
\gamma\overline{\gamma}=1\text{ and }\gamma\neq 1,-\beta\overline{\beta}.
\] But the equation $\gamma\overline{\gamma}=\gamma^{q+1}=1$ has $q+1\geq 4$ roots and thus we can avoid the two values $1,-\beta\overline{\beta}$.
\end{proof}

\begin{remark}\label{remark:QR_for_odd_q}
If $q$ is odd and $q\geq 5$, then we may choose $x$ in Proposition \ref{prop:a_quasi_reflection} to be a proper reflection, i.e., to have $\gamma=-1$. In fact, with this choice of $\gamma$, the restrictions in Equations \eqref{equ:QR-beta_gamma_basic} and \eqref{equ:QR-beta_gamma_advanced} amount to choose $\beta$ such that
\[
\beta\neq 0, \beta\overline{\beta}\neq 1,-1.
\]
As the codomain of the norm map has size $q-1\geq 4$, this choice is possible.
\end{remark}

\subsection{Symmetric group and certain subsets of it}\label{subsection:symmetric_groups} We denote by $S_n$ the symmetric group on $n$ symbols, i.e., the group of permutations of the set $\{1,\ldots,n\}$. We write $s_i$ for the transposition $(i,i+1)$, $i=1,\ldots,n-1$, and we recall that 
\begin{equation}\label{equ:relations_S_n}
s_is_{i+1}s_i=s_{i+1}s_is_{i+1}\text{ , }s_is_j=s_js_i\text{ if $|i-j|>1$, }s_i^2=1.
\end{equation}
Recall that the sign map $\sgn\colon S_n\to \{-1,+1\}$ is described as follows,
\begin{equation}\label{equ:sign_map_symmetric_group}
\sgn(s)=\begin{cases} +1,&\text{if $s$ is an even permutation,}\\
-1,&\text{if $s$ is an odd permutation,}
\end{cases}
\end{equation}
and that an element of  $S_n$ is called even (resp. odd) if it can be expressed as a product of an even (resp. odd) number of generators $s_i$'s. Below we present a canonical form for elements of $S_n$ in which $s_{n-1}$ appears at most once, see \cite[Section 4.1, Lemma 4.3, Exercise 4.1.3]{KasselTuraev}.

\begin{prop}[Normal form in $S_n$]\label{prop:normal_form_Sn}
For each permutation $w\in S_n$, there exist unique elements $w_i\in \{1,s_i,s_is_{i-1},\ldots,s_is_{i-1}\cdots s_2s_1\}$ for $i=1,\ldots,n-1$ such that
\[
w=w_1w_2\cdots w_{n-1}.
\]
\end{prop}

The next subsets of $S_n$ will play a crucial role in Section \ref{section:PGU}.

\begin{defn}\label{def:Sn+-partialSn+-}
We define $S^+_n\subseteq S_n$ ($S^-_n\subseteq S_n$) as the subset of permutations whose normal form (Proposition \ref{prop:normal_form_Sn}) have $w_1=s_1$ ($w_1=1$). In addition, we define $\partial S^+_n\subseteq S^+_n$ ($\partial S^-_n\subseteq S^-_n$) as the subset of permutations $w\in S^+_n$ ($w\in S^-_n$) such that $ws_j\in S^-_n$ ($ws_j\in S^+_n$) for some $1\leq j\leq n-1$.
\end{defn}

\begin{exa}\label{exa:Sn+-partialSn+-}
For $n=2$, we have
\[
S_2=\{1,s_1\}\text{, }S^+_2=\{s_1\}\text{, }S^-_2=\{1\}\text{, }\partial S^+_2=\{s_1\}\text{, and }\partial S^-_2=\{1\},
\]
and, for $n=3$, we have 
\[
S_3=\{1,s_2,s_2s_1,s_1,s_1s_2,s_1s_2s_1\}\text{, }S_3^+=\{s_1,s_1s_2,s_1s_2s_1\}\text{, }S_3^-=\{1,s_2,s_2s_1\}
\]
and
\[
\partial S_3^+=\{s_1,s_1s_2s_1\}\text{, }\partial S_3^-=\{1,s_2s_1\}.
\]
\end{exa}

The next result exhibits intriguing combinatorial properties of the subsets of $S_n$ defined above. These features will be interpreted geometrically in Section \ref{section:PGU}, explaining also the naming chosen in Definition \ref{def:Sn+-partialSn+-}; see Remark \ref{rmk:why_x_and_Y-+}. Note that point (b) below implies that the value $j$ satisfying the hypothesis there, as well as the value $j$ in Definition \ref{def:Sn+-partialSn+-}, if exists, is unique.

\begin{prop}\label{prop:properties_of_T} The following hold,
\begin{enumerate}
\item[(a)] $S_n=S_n^-\cup S_n^+$, $S_n^-\cap S_n^+=\emptyset$, $S_n^+=s_1S_n^-$, and $\partial S^+_n=s_1\partial S^-_n$.
\item[(b)] If $w\in \partial S^-_n$ with $ws_j\in S^+_n$, then $w(\{j,j+1\})=\{1,2\}$.
\item[(c)] If $w\in S_n^-$ then, for all $2\leq k\leq n$, we have
\begin{equation}\label{equ:properties_of_T_part_d}
w(j)=j\text{ for $k+1\leq j\leq n$ }\Rightarrow w(k)\neq 1.
\end{equation}
\end{enumerate}
\end{prop}
\begin{proof}
The first three parts of point (a) clearly follow from Proposition \ref{prop:normal_form_Sn}. On the other hand, by Definition \ref{def:Sn+-partialSn+-}, an element $s_1w\in S^+_n$ with $w\in S^-_n$ belongs to $\partial S^+_n$ if and only if $s_1ws_j\in S_n^-$ for some $1\leq j\leq n-1$, and an element $w\in S^-_n$ belongs to $\partial S^-_n$ if and only if $s_1ws_j\in S_n^-$ for some $1\leq j\leq n-1$. Hence, the last part of point (a) follows. Moreover, the hypothesis that $w\in \partial S_n^-$ in point (b), can be stated as follows, where we write $w=w_2\cdots w_{n-1}$ by Proposition \ref{prop:normal_form_Sn}, 
\begin{equation}\label{equ:properties_of_T_hypothesis_b_again}
s_1w_2\cdots w_{n-1}s_j\in S^-_n,
\end{equation}
and we prove point (b) by induction on $n$. The cases $n=2,3$ follow from Example \ref{exa:Sn+-partialSn+-} together with the following fact,
\[
(1)(\{1,2\})=(s_2s_1)(\{2,3\})=\{1,2\}.
\]
Thus, we assume that $n\geq 4$ and that point (b) is true for lesser values of $n$. We study $w_{n-1}$ and $j$ in Equation \eqref{equ:properties_of_T_hypothesis_b_again}: If $w_{n-1}=1$ and $j=n-1$, then
\[
s_1w_2\cdots w_{n-1}s_{n-1}=s_1w_2\cdots w_{n-2}s_{n-1}
\]
is in normal form and does not belong to $S^-_n$. If $w_{n-1}=1$ and $j<n-1$, then 
\[
s_1w_2\cdots w_{n-2}w_{n-1}s_{j}=s_1w_2\cdots w_{n-2}s_{j}\in S_{n-1}^-,
\]
and we may apply, by induction, that point (b) is true for $n-1$ and $w'=w_2\cdots w_{n-2}\in S_{n-1}^-$. Thus, $w'(\{j,j+1\})=\{1,2\}$, and then $w(\{j,j+1\})=w'(\{j,j+1\})=\{1,2\}$. Hence, we can assume that $1\neq w_{n-1}=s_{n-1}\cdots s_k$ with $1\leq k\leq n-1$. If $j<k-1$, then $[w_{n-1},s_j]=1$ by Equation \eqref{equ:relations_S_n}, and
\[
s_1w_2\cdots w_{n-2}w_{n-1}s_j=s_1w_2\cdots w_{n-2}s_jw_{n-1}\in S_{n-1}^-w_{n-1}.
\]
Thus, by the inductiom hypothesis for $n-1$ applied to $w'=w_2\cdots w_{n-2}\in S_{n-1}^-$, we have that $w'(\{j,j+1\})=\{1,2\}$, and then $w(\{j,j+1\})=w'(\{j,j+1\})=\{1,2\}$. If $j=k-1$, then
\[
s_1w_2\cdots w_{n-2}w_{n-1}s_{k-1}=s_1w_2\cdots w_{n-2}s_{n-1}\cdots s_ks_{k-1}
\]
is in normal form and does not belong to $S^-_n$. Analogously, if $j=k$, then
\[
s_1w_2\cdots w_{n-2}w_{n-1}s_{k}=s_1w_2\cdots w_{n-2}s_{n-1}\cdots s_ks_k=s_1w_2\cdots w_{n-2}s_{n-1}\cdots s_{k+1}
\]
is in normal form and not in $S^-_n$. Finally, if $k+1\leq j\leq n-1$, then
\begin{align*}
s_1w_2\cdots w_{n-2}w_{n-1}s_{j}&=s_1w_2\cdots w_{n-2}s_{n-1}\cdots s_ks_j\\
&=s_1w_2\cdots w_{n-2}s_{n-1}\cdots s_js_{j-1}s_j\cdots s_k\\
&=s_1w_2\cdots w_{n-2}s_{n-1}\cdots s_{j-1}s_{j}s_{j-1}\cdots s_k\\
&=s_1w_2\cdots w_{n-2}s_{j-1}s_{n-1}\cdots s_{j}s_{j-1}\cdots s_k\\
&=s_1w_2\cdots w_{n-2}s_{j-1}w_{n-1}\in S_{n-1}^-w_{n-1},
\end{align*}
where we have used Equation \eqref{equ:relations_S_n}. Thus, by the induction hypothesis for $n-1$ applied to $w'=w_2\cdots w_{n-2}\in S_{n-1}^-$, we have $w'(\{j-1,j\})=\{1,2\}$, and then $w(\{j,j+1\})=w'(\{j-1,j\})=\{1,2\}$.

Finally, for point (c), and if we write $w=w_1w_2\cdots w_{n-1}$, it is easy to see that the hypothesis in Equation \eqref{equ:properties_of_T_part_d} for $2\leq k\leq n$ is equivalent to $w_{n-1}=\cdots=w_k=1$. Moreover, if this condition holds, then $w(k)=1$ if and only if $w_{k-1},\ldots,w_2,w_1$ are all non-trivial, so that point (c) follows.
\end{proof}

\section{Topological preliminaries.} \label{section:topological_preliminaries} 

We regard the topological realization $|\A_p(G)|$ as a simplicial complex, so that it has one ordered $n$-simplex for each chain of length $n$ in $\A_p(G)$,
\[
P_0<\cdots<P_n,\text{ where $P_i\in \A_p(G)$ for $0\leq i\leq n$.}
\]
We note that the maximal faces of such an $n$-simplex are the ordered $(n-1)$-simplices
\[
P_0<\cdots<\widehat{P_i}<\cdots<P_n
\]
for $i=0,\ldots,n$, and that vertices of $|\A_p(G)|$ correspond to subgroups $P_0\in \A_p(G)$.

For a \emph{unital commutative ring} $R$, we may study the reduced homology of the topological space $|\A_p(G)|$ with trivial coefficients in $R$, denoted as $\widetilde H_*(|\A_p(G)|;R)$, via simplicial homology. Thus, these homology groups are then obtained as the homology of the augmented simplicial chain complex $C_*(|\A_p(G)|;R)$, that we next describe: $C_n(|\A_p(G)|;R)$ is the free $R$-module with basis the ordered $n$-simplices $P_0<\ldots <P_n$, and the differential 
\[
C_n(|\A_p(G)|;R)\stackrel{d}\longrightarrow C_{n-1}(|\A_p(G)|;R)
\]
is given, for $n\geq 1$, by 
\[
d=\sum_{j=0}^n (-1)^id_i,
\]
where $d_i$ is the $R$-linear map with 
\[
d_i(P_0<\cdots <P_n)=P_0<\cdots<\widehat{P_i}<\cdots<P_n,
\]
and, for $n=0$, by $d(P_0)=1\in R$ for $P_0\in \A_p(G)$, where $C_{-1}(|\A_p(G)|;R)=R$.

\subsection{Some homology classes.} In this subsection, we introduce a method to explicitly construct certain classes in the reduced homology groups of $|\A_p(G)|$, following \cite{Diaz2016} and \cite{DiazMazza2020}. We start with the following notions of tuple and signature, which will be employed to algebraically manipulate certain barycentric subdivision (see Definition \ref{def:ZEa} and Remark \ref{rmk:barycentric_subdivision}) as well as  the boundary map of linear combinations of these subdivisions (see Theorem \ref{thm:noncontractiblemoregeneral}).

\begin{defn}\label{def:sequences}
Let $r$ be a positive integer. We define a \emph{tuple for $r$} to be an ordered sequence of $r-1$ elements $\ii=[i_1,\ldots,i_{r-1}]$ with $1\leq i_j\leq r$
and no repetition. By $\tuples^r$ we denote the set of all
tuples for $r$ and, for $\ii=[i_1,\ldots,i_{r-1}]\in \tuples^r$, we define
the {\em signature of $\ii$} as
\[
\sgn(\ii)=(-1)^{n+m}\text{, where}
\]
\begin{itemize}
\item $n$ is the number of transpositions
we need to apply to the tuple $\ii$ to rearrange it in increasing order
$[j_1,\ldots,j_{r-1}]$, and\item $m$ is the number of positions in which
$[j_1,\ldots,j_{r-1}]$ differ from $[1,\ldots,r-1]$.\end{itemize}
\end{defn}

Note that in Definition~\ref{def:sequences}, the number $n$
of transpositions is not uniquely defined, but its parity is. For instance, if $\ii=[1,4,2]\in \tuples^4$, then $n=1$, since we need to apply $(2,4)$
to reorder $\ii$ as $[1,2,4]$, and $m=1$, since $[1,2,4]$ differs from
$[1,2,3]$ only in one place. Thus $\sgn(\ii)=1$.

\begin{prop}\label{prop:right_signature}
Let $\ii=[i_1,\ldots,i_{r-1}]$, $\jj=[j_1,\ldots,j_{r-1}]\in \tuples^r$ and define $1\leq i,j\leq r$ by $\{i_1,\ldots,i_{r-1},i\}=\{j_1,\ldots,j_{r-1},j\}=\{1,\ldots,r\}$. Let $\sigma$ be the permutation on $r$ letters with $\sigma(j_l)=i_l$ for $i=1,\ldots,r-1$ and with $\sigma(j)=i$. Then $\sgn(\sigma)=\sgn(\jj)\sgn(\ii)$.
\end{prop}
\begin{proof}
Define $\widehat \sigma_\ii$ as the permutation on $r$ letters that arranges $[i_1,\ldots,i_{r-1}]$ in increasing order $[1,2,\ldots,i-1,i+1,\ldots,r]$. Then $\sgn(\ii)=\sgn(\widehat\sigma_\ii)(-1)^{r-i}$ by Definition \ref{def:sequences}. Define $\sigma_\ii$ as the permutation on $r$ letters that arranges $[i_1,\ldots,i_{r-1},i]$ in increasing order $[1,\ldots,i-1,i,i+1,\ldots,r]$. Then we clearly have
\[
\sigma_\ii=(r, i)\circ(r-1,i)\circ\cdots\circ(i+1,i)\circ \widehat \sigma_\ii.
\]
Hence, $\sgn(\sigma_\ii)=\sgn(\widehat\sigma_\ii)(-1)^{r-i}=\sgn(\ii)$.
Analogously, we define $\sigma_\jj$ with $\sgn(\jj)=\sgn(\sigma_\jj)$. Then it is clear that 
$\sigma_\ii\sigma\sigma^{-1}_\jj=1$ and hence the result.
\end{proof}

\begin{defn}\label{def:sequencesubspace}
Let $E=\langle e_1,\dots,e_r \rangle$ be an elementary abelian $p$-subgroup of
$G$ of rank $r$. For each tuple $\ii=[i_1,\ldots,i_{r-1}]\in \tuples^r$ and each $0\leq l<r$, set
\[
E_{[i_1,\ldots,i_l]}=\langle e_k\ \colon \ 1\leq k\leq r,\ k\not\in\{i_1,\ldots,i_l\}\rangle,
\]
for the subgroup of $E$ generated by the $e_k$'s not in the first $l$ entries of $\ii$.
\end{defn}

In the definition above, for $l=0$, we have $E_{\emptyset}=E$, for $l=1$, we obtain the hyperplane $E_{[i_1]}$ of $E$, that we also denote by $E_{i_1}$, and, for $l=r-1$,  $E_{[i_1,\ldots,i_{r-1}]}$ is a cyclic subgroup of order $p$. Next, we define certain simplicial chains in $C_*(|\A_p(G)|;R)$.

\begin{defn}\label{def:ZEa}
Let $E=\langle e_1,\dots,e_r \rangle$ be an elementary abelian $p$-subgroup of $G$ of rank $r$. For $\ii=[i_1,\ldots,i_{r-1}]\in \tuples^r$, we define the
$(r-1)$-simplex 
\[
\sigma_{\ii}=\big(E_{[i_1,\ldots,i_{r-1}]}<E_{[i_1,\dots,i_{r-2}]}<
\cdots<E_{i_1}<E\big)\in|\A_p(E)|,
\]
and the chain
\[
\barsub_E=\sum_{\ii\in \tuples^r}\sgn(\ii)\sigma_{\ii}
\qbox{in}C_{r-1}(|\A_p(E)|;R).
\]
\end{defn}

Its differential $d(\barsub_E)\in C_{r-2}(|\A_p(E)|;R)$ is described in the next result, cf.  \cite[Proposition 3.2]{Diaz2016}, \cite[Proposition 3.2]{DiazMazza2020}.
\begin{defn}\label{def:flagtau}
With the notation above and for $\ii=[i_1,\ldots,i_{r-1}]\in \tuples^r$, we define the
$(r-2)$-simplex 
\[
\tau_{\ii}=\big(E_{[i_1,\ldots,i_{r-1}]}<E_{[i_1,\dots,i_{r-2}]}<\cdots<E_{i_1}\big).
\]
\end{defn}

\begin{prop}\label{prop:dz}
With the above notation,
\[
d(\barsub_E)=(-1)^{r-1}\sum_{\ii\in \tuples^r}\sgn(\ii)\tau_{\ii}.
\]
\end{prop}

\begin{proof}  
Recall that $d(\barsub_E)=\sum_{j=0}^{r-1}(-1)^jd_j(\barsub_E)$, where $d_j$ removes the $(j+1)$-th leftmost term of a given $(r-1)$-simplex. Suppose that $d_k(\sigma_{\ii})=d_l(\sigma_{\jj})$ for some
$\ii,\jj\in \tuples^r$ and $0\leq k,l\leq r-1$, i.e.,
\begin{multline*}
\big(E_{[i_1,\ldots,i_{r-1}]}<\dots<E_{[i_1,\dots,i_{r-k}]}<
E_{[i_1,\dots,i_{r-k-2}]}<\dots<E_{i_1}<E\big)=\\
=\big(E_{[j_1,\ldots,j_{r-1}]}<\dots<E_{[j_1,\dots,i_{r-l}]}<
E_{[j_1,\dots,j_{r-l-2}]}<\dots<E_{j_1}<E\big).
\end{multline*}
Comparing the sizes of the elementary abelian $p$-subgroups occurring in these two chains, it is clear that we must have $k=l$. If $0<k=l<r-1$, then the tuples $\ii$ and $\jj$ are
identical but for $\{ i_{r-k-1}, i_{r-k} \}=\{ j_{r-k-1}, j_{r-k}\}$. Thus, either $\ii=\jj$, or $\sgn(\ii)=-\sgn(\jj)$ by Proposition \ref{prop:right_signature}. In the latter case, the corresponding summands $\sgn(\ii)d_k(\sigma_\ii)$ and $\sgn(\jj)
d_l(\sigma_\jj)$ add up to zero. Assume now that $k=l=0$. Then $[j_1,\dots,j_{r-2}]=[i_1,\dots,i_{r-2}]$ and either $\ii=\jj$, or $j_{r-1}\neq i_{r-1}$ and $\sgn(\ii)=-\sgn(\jj)$ by  Proposition \ref{prop:right_signature} again. In the latter case, the two terms cancel each other out again. Finally, if $k=l=r-1$, the tuples $\ii$ and $\jj$ are identical and the terms contribute to the sum in the statement of the lemma. 
\end{proof}

\begin{remark}\label{rmk:barycentric_subdivision}
The chain $\barsub_E$ represents the barycentric subdivision of an $(r-1)$-simplex $\sigma$. The $(r-1)$-simplex $\sigma_\ii$ for each tuple $\ii\in \tuples^r$ corresponds to one piece of the subdivision of $\sigma$, and the $(r-2)$-simplex $\tau_\ii$ to one piece in the subdivision of the boundary of $\sigma$. For instance, for $r=3$, we have the subdivision of a triangle shown in Figure \ref{fig:barycentric_subdivision}. For each $\ii\in \tuples^2=\{[1,2],[2,1],[1,3],[3,1],[2,3],[3,2]\}$, $\sigma_\ii$ is one of the $6$ small triangles, and $\tau_\ii$ is one of the $6$ segments in the boundary of the large triangle. The $7$ vertices correspond to the subgroups $E_{[i_1,\ldots,i_l]}$ for $l=0,1,2$ and, for $l=0$, $E$ is located at the barycenter of the triangle.
\begin{figure}[h!]
\centering

    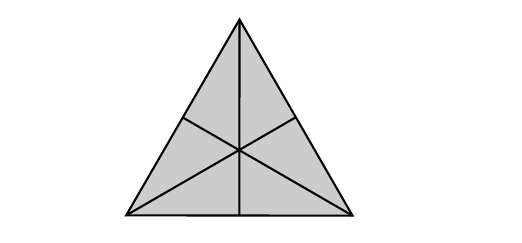

\caption{Barycentric subdivision of a $2$-simplex.}\label{fig:barycentric_subdivision}
\end{figure}

\end{remark}

Recall that $G$ acts by conjugation on $C_{r-1}(|\A_p(G)|;R)$ and that $C_{r-1}(|\A_p(E)|;R)\subseteq C_{r-1}(|\A_p(G)|;R)$. Then, for $x\in G$ and $a\in R$, the element
\[
\ls x(a\barsub_E)=a\sum_{\ii\in \tuples^r}\sgn(\ii)
\ls x{\sigma_{\ii}}=
a\sum_{\ii\in \tuples^r}\sgn(\ii)
\big(\ls xE_{[i_1,\ldots,i_{r-1}]}<\dots<\ls xE_{i_1}<\ls xE\big)
\]
belongs to $C_{r-1}(|\A_p({}^xE)|;R)\subseteq C_{r-1}(|\A_p(G)|;R)$. 
\begin{defn}\label{def:ZGXA}
Let $X\subseteq G$ be a non-empty subset of $G$ and let $h\colon X\to R$ be a function. Define the chain in $C_{r-1}(|\A_p(G)|;R)$,
\[
\barsub_{E,X,h}=\sum_{x\in X}\ls{x}(h(x)\barsub_E)=
\sum_{x\in X}h(x)\sum_{\ii\in \tuples^r}
\sgn(\ii)\ls{x}{\sigma_{\ii}}.
\]
\end{defn}

Now, by Proposition \ref{prop:dz},
\[
d(\barsub_{E,X,h})=(-1)^{r-1}\sum_{x\in X}h(x)\sum_{\ii\in \tuples^r}
\sgn(\ii)\ls{x}{\tau_{\ii}}.
\]
Hence, given $x\in X$ and $\ii\in \tuples^r$, the coefficient of
$\ls{x}{\sigma_{\ii}}$ in $\barsub_{E,X,h}$ and the coefficient of $\ls{x}{\tau_{\ii}}$ in $d(\barsub_{E,X,h})$ are, respectively:
\begin{align*}
C(x,\ii)&=\sum_{(y,\jj)\in\C(x,\ii)}h(y)\sgn(\jj)\text{, }\quad\text{and}
\\ 
D(x,\ii)&=(-1)^{r-1}\sum_{(y,\jj)\in\D(x,\ii)}h(y)\sgn(\jj),
\end{align*}
where
\begin{align*}
\C(x,\ii)&=\{(y,\jj)\in X\times \tuples^r~|~\ls{y}\sigma_\jj=\ls{x}\sigma_\ii\}\qbox{and}
\\ 
\D(x,\ii)&=\{(y,\jj)\in X\times \tuples^r~|~\ls{y}\tau_\jj=\ls{x}\tau_\ii\}.
\end{align*}
If we further assume that $E$ is a \emph{maximal} elementary abelian
$p$-subgroup of $G$, we immediately obtain the following result, see also \cite[Theorem 3.3]{DiazMazza2020}.

\begin{thm}\label{thm:noncontractiblemoregeneral}
Let $E=\langle e_1,\dots,e_r\rangle$ be a \emph{maximal} elementary
abelian $p$-subgroup of rank $r$ of the group $G$. Consider a subset $X\subseteq G$ and $h\colon X\to R$ satisfying that
\begin{enumerate}[\hspace{1cm}(a)]
\item $C(x,\ii)\neq 0$ for some $x\in X$ and some
$\ii\in \tuples^r$,
\item  $D(x,\ii)=0$ for all $x\in X$ and all
$\ii\in \tuples^r$.
\end{enumerate}
Then
\[
0\neq [\barsub_{E,X,h}]\in \widetilde H_{r-1}(|\A_p(G)|;R).
\]
\end{thm}

The maximality condition on $E$ is needed to ensure that the non-trivial cycle $[\barsub_{E,X,h}]$ is not a boundary. This maximality condition holds, for instance, if $E$ attains the maximum possible rank among the elementary $p$-subgroups of $G$, i.e., if $r=\rk_p(G)$. The conclusion of the theorem implies that  $|\A_p(G)|$ is not contractible and, in particular, that Quillen's conjecture holds for $G$. If $r=\rk_p(G)$, then the conclusion of the theorem and the fact that top dimension integral homology is a free abelian group, show that the $\Q\D_p$ property holds, see \cite[Definition p.474]{AS1993}.
\begin{defn} [$\Q\D_p$] \label{def:QDp}
The finite group $G$ with $r=\rk_p(G)$ has the \emph{Quillen dimension property for $p$}, $\Q\D_p$ for short, if 
\[
\widetilde H_{r-1}(|\A_p(G)|;\QQ)\neq 0.
\]
\end{defn}
Note also that the hypotheses of Theorem \ref{thm:noncontractiblemoregeneral} imply that $O_p(G)=1$ because of Quillen's ``conical contractibility'' \cite[Proposition 2.4]{Quillen1978}. Next, we describe a set of general group theoretical conditions that imply item (a) or (b) of Theorem \ref{thm:noncontractiblemoregeneral}, and some of the upcoming arguments are based on the following elementary fact. 

\begin{lem}\label{lem:sum_set_with_involution_is_0}
Let $\D$ be a finite set, $h\colon \D\to R$ a function, and $\Omega$ a non-trivial involution of $\D$, i.e., a map $\Omega\colon \D\to \D$ such that $\Omega^2(y)=y\neq \Omega(y)$ for all $y\in \D$,  and assume that $h(y)+h(\Omega(y))=0$ for all $y\in \D$. Then $\sum_{y\in \D} h(y)=0$.
\end{lem}
\begin{proof}
Partition the set $\D$ into the subsets $\{y,\Omega(y)\}$.
\end{proof}

In the statement below, for a subset $X$ and an element $y$ of a group $G$, we write $Xy$ for the subset $\{xy\mid x\in X\}$, and we set, for $\ii,\jj\in \tuples^r$ with $\jj=[j_1,\ldots,j_{r-1}]$ and for elements $\{y_1,\ldots,y_r\}\subseteq G$,
\begin{align}\label{equ:D(x,i,j)}
\D(x,\ii,\jj)&=\{y\in X~|~\ls{y}\tau_\jj=\ls{x}\tau_\ii\}\text{, }\D^*(x,\ii,\jj)=\D(x,\ii,\jj)\cap (Xy_{j_1}\cup Xy_{j_1}^{-1}),\\
D(x,\ii,\jj)&=\sum_{y\in \D(x,\ii,\jj)} h(y)\text{, and }D^*(x,\ii,\jj)=\sum_{y\in \D^*(x,\ii,\jj)} h(y).\nonumber
\end{align}
Lemma \ref{lem:sum_set_with_involution_is_0}, Equation \eqref{equ:D(x,i,j)}, and items (b,b2,c) below will be used in Section \ref{section:PGU}. We stress that, in the statement below, the elements $y_i$ do not have to belong to $X$.

\begin{lem}\label{lem:hypotheses_for_(a)(b)}
Let $E=\langle e_1,\dots,e_r\rangle$ be a \emph{maximal} elementary
abelian $p$-subgroup of rank $r$ of the group $G$, $\{y_1,\ldots,y_r\}\subseteq G$, $X\subseteq G$, and $h\colon X\to R^*$. Then
\begin{enumerate}
\item[(a)] If for all $x',x\in X$ and all $n\in N_G(E)$ with $x'=xn$, we have $x'=x$, then $G$, $E$, $X$, $h$ satisfy hypothesis (a) of Theorem \ref{thm:noncontractiblemoregeneral}.
\end{enumerate}
Moreover, if $\ii,\jj\in \tuples^r$, $\jj=[j_1,\ldots,j_{r-1}]$, and we assume that 
\begin{enumerate}
\item[(b)] $y_{j_1}\in C_G(E_{j_1})$ and, if $y\in X\cap Xy_{j_1}^\epsilon$, $\epsilon\in \{-1,1\}$, then $h(y)+h(yy_{j_1}^{-\epsilon})=0$, 
\end{enumerate}
then the following hold,
\begin{enumerate}
\item[(b1)] If $X\cap Xy_{j_1}\cap Xy_{j_1}^{-1}=\emptyset$, then $D^*(x,\ii,\jj)=0$.
\item[(b2)] If $y_{j_1}$ has order $2$, then $D^*(x,\ii,\jj)=0$.
\end{enumerate}
In addition, the following statements hold,
\begin{enumerate}
\item[(c)] $D(x,\ii)=(-1)^{r-1}\sum_{\jj\in \tuples^r}\sgn(\jj)D(x,\ii,\jj)$.
\item[(c1)] If for all $1\leq j_1\leq r$, $X\subseteq (Xy_{j_1}\cup Xy_{j_1}^{-1})$ and the hypotheses of (b) and (b1) hold, then $G$, $E$, $X$, $h$ satisfy hypothesis (b) of Theorem \ref{thm:noncontractiblemoregeneral}.
\item[(c2)] If for all $1\leq j_1\leq r$, $X\subseteq Xy_{j_1}$ and the hypotheses of (b) and (b2) hold, then  $G$, $E$, $X$, $h$ satisfy hypothesis (b) of Theorem \ref{thm:noncontractiblemoregeneral}.
\end{enumerate}
\end{lem}
\begin{proof}
We consider the chain $\barsub_{E,X,h}$ of Definition \ref{def:ZGXA} and start with item (a): Consider $x\in X$ and $\ii\in \tuples^r$. If $(x',\jj)\in \C(x,\ii)$ then $\ls{x'}\sigma_\jj=\ls{x}\sigma_\ii$ and, in particular, $\ls{x'}E=\ls{x}E$. Hence $x^{-1}x'\in N_G(E)$ and, by the hypothesis in (a), we must have $x'=x$. Thus, we have $\ii=\jj$, $\C(x,\ii)=\{(x,\ii)\}$, and $C(x,\ii)=h(x)\neq 0$.

For item (b1), note that, by definition, $\D^*(x,\ii,\jj)=\D^+(x,\ii,\jj)\cup \D^-(x,\ii,\jj)$, where 
\[
\D^+(x,\ii,\jj)=\D^*(x,\ii,\jj)\cap Xy_{j_1}\text{ and }\D^-(x,\ii,\jj)=\D^*(x,\ii,\jj)\cap Xy_{j_1}^{-1},
\]
and that, by the hypothesis in point (b1),
\[
\D^+(x,\ii,\jj)\cap \D^-(x,\ii,\jj)\subseteq X\cap Xy_{j_1}\cap Xy_{j_1}^{-1}=\emptyset.
\]
Now, by the first hypothesis in point (b),  $y_{j_1}$ belongs to $C_G(E_{j_1})=C_G(\tau_\jj)$. In particular, we have a non-trivial involution $\D^*(x,\ii,\jj)\to \D^*(x,\ii,\jj)$, given by $y\mapsto yy_{j_1}^{-1}$ if $y\in \D^+(x,\ii,\jj)$ and by $y\mapsto yy_{j_1}$ if $y\in \D^-(x,\ii,\jj)$. Then we are done by Lemma \ref{lem:sum_set_with_involution_is_0} and by the second hypothesis in point (b). For item (b2), the arguments are identical but for considering the non-trivial involution $\D^*(x,\ii,\jj)\to \D^*(x,\ii,\jj)$ given by $y\mapsto yy_{j_1}$.

Item (c) follows from the decomposition $\D(x,\ii)=\cup_{\jj\in \tuples^r} \D(x,\ii,\jj)\times \{\jj\}$. Finally, for items (c1) and (c2), note that, in these cases, we have $\D^*(x,\ii,\jj)=\D(x,\ii,\jj)$ and $D^*(x,\ii,\jj)=D(x,\ii,\jj)$,  so that the conclusions follow directly from items (b,b1,c) or (b,b2,c), respectively,
\end{proof}

Note that, if $X$ is a non-trivial subgroup of $G$ with $\{y_1,\ldots,y_r\}\subseteq X$, then the condition of Lemma \ref{lem:hypotheses_for_(a)(b)}(a) takes the following simplified form,
\[
X\cap N_G(E)=\{1\},
\]
the hypothesis in Lemma \ref{lem:hypotheses_for_(a)(b)}(b1) does not hold, and the conditions involving $X$ in Lemma \ref{lem:hypotheses_for_(a)(b)}(c1,c2) are valid. This situation that $X$ is a subgroup of $G$ does not arise in Section \ref{section:PGU}, but it might be useful in future applications. Below we give an example with $X$ a subgroup of $G$.
\begin{exa}
For the symmetric group on $7$ letters, $G=S_7$, $p=3$, we have $O_3(G)=1$, $\rk_3(G)=2$, and we set 
\[
e_1=(1,2,3)\text{, }e_2=(4,5,6)\text{, }y_1=(3,7)\text{, and }y_2=(6,7).
\]
Then $E=\langle e_1,e_2\rangle$ is an elementary abelian $p$-subgroup of $p$-rank $2$, and we choose  $X$ as the following \emph{subgroup} of $G$,
\[
X=\langle\{y_1,y_2\}\rangle\cong S_3.
\]
We define the function $h$ as the sign map $h\colon S_3\to \{-1,+1\}\subseteq \ZZ=R$, i.e., by
\[
h(1)=h(y_1y_2)=h(y_2y_1)=+1\text{ and }h(y_1)=h(y_2)=h(y_1y_2y_1)=-1.
\]
Then it is straightforward that the hypotheses in items (a,b,b2,c2) of Lemma  \ref{lem:hypotheses_for_(a)(b)} hold, so that Theorem \ref{thm:noncontractiblemoregeneral} gives  $\widetilde H_1(|\A_3(G)|;\ZZ)\neq 0$ and that $\Q\D_3$ holds for $G$. In Figure \ref{fig:triangulationS7r=2}, the non-zero homology cycle constructed is depicted, with bullet points and squares corresponding, respectively, to $p$-rank 1 and 2 elementary abelian $p$-subgroups. See also the barycentric subdivisions in Remark \ref{rmk:barycentric_subdivision}.
\begin{figure}[h!]
\centering

    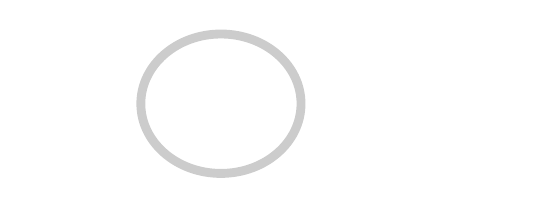

\caption{Triangulation of the sphere $S^1$ arising from $\A_3(S_7)$.}
\label{fig:triangulationS7r=2}
\end{figure}
 As explained in \cite[Remark 4.3]{DiazMazza2020}, in this case we cannot expect to find elements $y_1$ and $y_2$ that commute.
\end{exa}

\section{Proof of Theorem A.}
\label{section:PGU}

In this section, we prove the Quillen dimension property for $\PGU_n(q)$ and its extensions by field automorphisms of order $p$, as in Theorem A of the Introduction. The techniques we employ are better described in terms of the symmetric group embedded into $\GU_n(q)$, see Definition \ref{def:yiY}, and quasi-reflections, see Subsection \ref{subsection:quasi-reflections}. We will apply Theorem \ref{thm:noncontractiblemoregeneral} and Lemmas \ref{lem:sum_set_with_involution_is_0} and \ref{lem:hypotheses_for_(a)(b)} to construct non-trivial homology classes in the appropriate dimensions; see Theorem \ref{thm:QDp_for_PGUn} for the case of $\PGU_n(q)$ and Theorem \ref{thm:QDp_for_PGUnfield} for the case of its extension by field automorphisms of order $p$. In general, there will be non-trivial identifications among simplices. See  Remark \ref{rmk:some_nonidentified_simplices} and Example \ref{example:n2_quasireflections} for more details. Unless stated otherwise, we assume throughout this section that $p$ is an odd prime dividing $q+1$ with $(p,q)\neq (3,2)$, and that $n\geq 2$.

We start with the case of $\PGU_n(q)$ and refer the reader to Subsection \ref{subsection:preliminaries_unitary_groups} for basic properties of this group, in particular, by Proposition \ref{prop:pextensions_PSUnq}, we have 
\[
\rk_p(\PGU_n(q))=n-1.
\]
Throughout this section, we write $Z=Z(\GU_n(q))$  and, for an element $x\in \GU_n(q)$ (respectively a subset $X\subseteq \GU_n(q)$),
we denote by $[x]$ (resp. $[X]$) the image of $x$ (resp. $X$) in the quotient $\PGU_n(q)$. Moreover, for a function $h\colon X\to \ZZ^*$,
we denote by $[h]\colon [X]\to \ZZ^*$ the map $[h]([x])=h(x)$, where we assume that $X\to [X]$ is an injection whenever we consider $[h]$. Finally, for an $r$-simplex of $|\A_p(\GU_n(q))|$,
\[
\sigma=P_0<\cdots<P_r,
\]
we denote by $[\sigma]$ the $r$-simplex of $|\A_p(\PGU_n(q))|$ given as follows, 
\[
[\sigma]=[P_0]<\cdots<[P_r],
\]
where we assume that the strict inequalities are preserved whenever we consider $[\sigma]$.  As the coming proofs are lengthy, we have split them into several intermediate results that incrementally present the different elements and arguments. We start by recalling the usual embedding of the symmetric group $S_n$ into $\GU_n(q)$.

\begin{defn}\label{def:yiY}
For $i=1,\ldots,n-1$, let $y_i$ be the usual permutation matrix of $\GU_n(q)$ that corresponds to the transposition $s_i=(i,i+1)\in S_n$, and denote by $Y=\langle y_1,\ldots,y_{n-1}\rangle$ the group of permutation matrices of $\GU_n(q)$, so that we have an isomorphism  $Y\to S_n$ given by $y_i\mapsto s_i$.
\end{defn}

In the rest of the paper, we will employ the isomorphism $Y\to S_n$ above without further notice. In particular, we define the subsets $Y^+, Y^-,\partial Y^+, \partial Y^-$ as the pre-images of the subsets $S_n^+,S_n^-,\partial S_n^+, \partial S^-_n$ by this isomorphism, respectively, see Definition \ref{def:Sn+-partialSn+-}. Next, we recall that the permutation matrix $y\in Y$ acts on diagonal elements of $T\leq \GU_n(q)$ as follows,
\begin{align}
\label{equ:Yacts_diagonal_elements}
\ls{y}\diag(\alpha_1,\ldots,\alpha_n) & =\diag(\alpha_{y^{-1}(1)},\ldots,\alpha_{y^{-1}(n)}).
\end{align}
In the next result, we present most of the ingredients needed in the proof of Theorem \ref{thm:QDp_for_PGUn}. In point (b) below, we recall that $E_i$ is the subgroup of $E$ spanned by the $e_k$ with $k\neq i$. In addition, item (c) implies that, first, every element of $X$ has a unique expression of the form $x^\delta y$ and, second, that the projection $X\to[X]$ is injective.

\begin{lem}\label{lem:EXforPGU}
If $n\geq 2$, $p$ is any prime, $q\geq 3$, and $\mu\in \FF^*_{q^2}$ has order $p$ and satisfies $\mu\overline{\mu}=1$, then there exist diagonal elements of order $p$, $e_1,e_2,e_3,\ldots,e_{n-1}$ in $T\leq \GU_n(q)$, such that, for $E=\langle e_1,e_2,\ldots,e_{n-1}\rangle$, we have,
\begin{enumerate}
\item[(a)] $[E]$ is an elementary abelian $p$-subgroup of $\PGU_n(q)$ of rank $n-1$, and
\item[(b)] $y_i\in C_{\GU_n(q)}(E_i)$ for $1\leq i\leq n-1$.
\end{enumerate}
Moreover, there exists an element $x\in \GU_n(q)$ such that, if we set $X=\{x^\delta y\mid \delta=0,1, y\in Y^-\}$, we have,
\begin{enumerate}
\item[(c)] The map $ \{0,1\}\times Y\to \PGU_n(q)$, $(\delta,y)\mapsto [x^\delta y]$, is injective, and
\item[(d)] For any diagonal element $t=\diag(\alpha_1,\ldots,\alpha_n)\in T$, we have that $\ls{x}{t}$ is non-diagonal if $\alpha_1\neq \alpha_2$ and that $\ls{x}t=t$ if $\alpha_1=\alpha_2$.
\end{enumerate}
\end{lem}
\begin{proof}
We define the elements $e_i$ of the statement as follows, 
\[
e_i=\diag(\mu,\ldots,\mu,1,\ldots,1),
\]
where $1\leq i\leq n-1$ and $\mu$ appears $i$ times, and we record the following observation,
\begin{equation}\label{equ:two_diagonal_entries_equal}
\text{the $j$-th and $(j+1)$-th diagonal entries of $e_i$ are equal if and only if $j\neq i$.}
\end{equation}

For point (a), assume that 
\[
e_1^{k_1}e_2^{k_2}\cdots e_{n-1}^{k_{n-1}}=z\in Z
\]
for some integers $k_i$. As all the elements $e_i$'s have last diagonal entry equal to $1$, we immediately deduce that $z=1$, and then it easily follows that $k_i=0\mod p$. Point (b) is an easy computation employing Equation \eqref{equ:two_diagonal_entries_equal} and the action of $Y$ on diagonal elements described in Equation \eqref{equ:Yacts_diagonal_elements}.

Next, we define the element $x$ of the statement as the quasi-reflection given by Proposition \ref{prop:a_quasi_reflection}, so that point (d) is immediate from the conclusions there. Regarding point (c), assume that $x^\delta y =zx^{\delta'}y'$ for $\delta,\delta'\in \{0,1\}$, $y,y'\in Y$, and $z\in Z$. Then $x^{\delta-\delta'}=zy'y^{-1}$, and we note that the right-hand side is a monomial matrix. But $x$, and hence also $x^{-1}$, is not a monomial matrix by Proposition \ref{prop:a_quasi_reflection}. Thus, we conclude that $\delta=\delta'$, that $y=zy'$, and that $y'=y$.
\end{proof}

Now we are ready to prove the $\Q\D_p$ property for $\PGU_n(q)$.

\begin{thm}\label{thm:QDp_for_PGUn}
If $n\geq 2$, $p$ is odd, $p$ divides $q+1$, and $(p,q)\neq (3,2)$, then we have $\widetilde{H}_{n-2}(|\A_p(\PGU_n(q))|;\ZZ)\neq 0$ and $\Q\D_p$ holds for $\PGU_n(q)$.
\end{thm}
\begin{proof}
We check the hypotheses of Theorem \ref{thm:noncontractiblemoregeneral} for $G=\PGU_n(q)$. 
Firstly, as $p\mid (q+1)$, there is an element $\mu$ of order $p$ in $\FF^*_{q^2}$ with $\mu\overline{\mu}=1$. In addition, as $p$ is odd, $p$ divides $q+1$, and $(p,q)\neq (3,2)$, we necessarily have $q\geq 4$. Then we may apply Lemma \ref{lem:EXforPGU} with this value of $\mu$ and obtain $e_1,e_2,e_3,\ldots,e_{n-1}$, $E$, $x$, and $X$, satisfying the conclusions there. As required in Theorem \ref{thm:noncontractiblemoregeneral}, $[E]$ is a maximal elementary abelian $p$-subgroup of $G$ by Lemma \ref{lem:EXforPGU}(a) and because $\rk_p(G)=n-1$ by Proposition \ref{prop:pextensions_PSUnq}. Consider now the following map $h\colon X\to \ZZ^*$,
\begin{equation}\label{equ:h_for_PGUn}
h(x^\delta y)=\sgn(y)(-1)^\delta,
\end{equation}
where the sign map was defined in Equation \eqref{equ:sign_map_symmetric_group}. Then we show below that
\[
[E]=\langle [e_1],\ldots,[e_{n-1}]\rangle\text{, }[X]\text{, and }[h]
\]
fulfil items (a) and (b) of Theorem \ref{thm:noncontractiblemoregeneral}. Note that $X\mapsto [X]$ is an injection by Lemma \ref{lem:EXforPGU}(c), so that the map $[h]\colon [X]\to \ZZ^*$ is well defined.

\textbf{Hypothesis $(a)$ of Theorem \ref{thm:noncontractiblemoregeneral} holds.}  We show that, for the tuple $\ii=[1,2,\ldots,n-2]\in \tuples^{n-1}$,  we have 
\[
\C([1],\ii)=\{([1],\ii)\},
\]
and this is enough for Theorem \ref{thm:noncontractiblemoregeneral}(a) to hold as the co-domain of $[h]$ is $\ZZ^*$. In turn, this is equivalent to show that, if $x^\delta y \in X$ and $\jj\in \tuples^{n-1}$ satisfy that
\[
[\ls{x^\delta y}\sigma_\jj]=[\sigma_\ii],
\]
then $\delta=0$, $y=1$, and $\jj=\ii$. If we write 
\[
\jj=[j_1,j_2,\ldots,j_{n-2}]
\]
with $\{1,\ldots,n-1\}=\{j_1,j_2,\ldots,j_{n-1}\}$, we are requiring that $[x^\delta y]$ conjugates the $(n-2)$-simplex $[\sigma_\jj]$, which is equal to
\[
\langle [e_{j_{n-1}}]\rangle<\langle [e_{j_{n-1}}],[e_{j_{n-2}}]\rangle<\cdots<\langle [e_{j_{n-1}}],\ldots,[e_{j_2}]\rangle<[E]=\langle [e_{j_{n-1}}],\ldots,[e_{j_2}],[e_{j_1}]\rangle,
\]
into the $(n-2)$-simplex $[\sigma_\ii]$, which equals
\[
\langle [e_{n-1}]\rangle<\langle [e_{n-1}],[e_{n-2}]\rangle<\cdots<\langle [e_{n-1}],\ldots,[e_2]\rangle<[E]=\langle [e_{n-1}],\ldots,[e_2],[e_1]\rangle.
\] 

Examining the two $(n-2)$-simplices above, we have that, for each $n-1\geq k\geq 2$,
\[
\ls{x^\delta y}e_{j_k}=z_k\cdot e_{n-1}^{m_{k,n-1}}\cdots e_k^{m_{k,k}},
\]
for some integers $m_{k,i}$ and some elements $z_k\in Z$. Conjugating by $x^{-\delta}$, and using, if $\delta=1$, Lemma \ref{lem:EXforPGU}(d) and Equation \eqref{equ:two_diagonal_entries_equal}, we obtain that
\begin{equation}\label{equ:n-2simplices_unfolded}
\ls{y}e_{j_k}=z_k\cdot e_{n-1}^{m_{k,n-1}}\cdots e_k^{m_{k,k}}.
\end{equation}
Based on this equation, we will prove inductively on $k=n,n-1,\ldots,2$, that
\begin{equation}\label{equ:induction_for_hypotehsis_a}
j_m=m\text{ for $m=n-1,n-2,\ldots,k$ and }y(m)=m\text{ for $m=n,\ldots,k+1$.}
\end{equation}

We note that the base case $k=n$ is vacuous and assume that Equation \eqref{equ:induction_for_hypotehsis_a} holds for $k$ with $n\geq k\geq 3$. Then  Equation \eqref{equ:properties_of_T_part_d} of Proposition \ref{prop:properties_of_T}, with the $k$ there equal to the $k$ here, shows that $2\leq y(k)\leq k$. On the other hand, we have, by hypothesis, that $j_{k-1}\leq k-1$. Then, checking against Equation \eqref{equ:n-2simplices_unfolded}, with the $k$ there equal to $k-1$ here, we find that the only solution is $y(k)=k$ and $j_{k-1}={k-1}$. Hence, Equation \eqref{equ:induction_for_hypotehsis_a} holds for $k-1$.

Now, Equation \eqref{equ:induction_for_hypotehsis_a} for $k=2$ gives $y(k)=k$ for $k=n,\ldots,3$ and $j_k=k$ for $k=n-1,\ldots,2$. Thus, $j_1=1$, $\jj=\ii$, and, from Equation \eqref{equ:properties_of_T_part_d} with the $k$ there equal to $2$, we obtain that $y(2)\neq 1$, i.e., that $y(2)=2$. Hence $y=1$ and we are left with proving that $\delta=0$ under the assumption that $j_1=1$. In fact, examining the right-most term of the $(n-2)$-simplices above, we obtain that
\[
\ls{x^\delta}e_1=z_1\cdot e_{n-1}^{m_{1,n-1}}\cdots e_2^{m_{1,2}}e_1^{m_{1,1}}
\]
for some integers $m_{1,i}$ and some $z_1\in Z$. If $\delta=1$, we get a contradiction with Lemma \ref{lem:EXforPGU}(d) and Equation \eqref{equ:two_diagonal_entries_equal}. Thus, $\delta=0$ and we are done.

\textbf{Hypothesis $(b)$ of Theorem \ref{thm:noncontractiblemoregeneral} holds.}  For tuples $\ii,\jj\in \tuples^{n-1}$, where $\jj=[j_1,\ldots,j_{n-2}]\in \tuples^{n-1}$, and $x^\delta y\in X$, define
\begin{align*}
\D^-([x^\delta y],\ii,\jj)&=\{[x^{\delta'}y']\in [X]~|~[\ls{x^{\delta'}y'}\tau_\jj]=[\ls{x^\delta y}\tau_\ii]\text{ and }y'y_{j_1}\in Y^-\},\\
\D^+([x^\delta y],\ii,\jj)&=\{[x^{\delta'}y']\in [X]~|~[\ls{x^{\delta'}y'}\tau_\jj]=[\ls{x^\delta y}\tau_\ii]\text{ and }y'y_{j_1}\in Y^+\},
\end{align*}
and denote by $D^-([x^\delta y],\ii,\jj)$ and $D^+([x^\delta y],\ii,\jj)$ the corresponding sums. Then, by the first and second parts of Proposition \ref{prop:properties_of_T}(a), we obtain that, see Equation \eqref{equ:D(x,i,j)},
\[
D([x^\delta y],\ii,\jj)=D^-([x^\delta y],\ii,\jj)+D^+([x^\delta y],\ii,\jj),
\]
and, by Lemma \ref{lem:hypotheses_for_(a)(b)}(c), it is enough to show that both summands above are zero.

Now, we obtain that $D^-([x^\delta y],\ii,\jj)=D^*([x^\delta y],\ii,\jj)$ from Lemma \ref{lem:EXforPGU}(c), and that $D^*([x^\delta y],\ii,\jj)=0$ from Lemma \ref{lem:EXforPGU}(b), Equation \eqref{equ:h_for_PGUn}, and Lemma \ref{lem:hypotheses_for_(a)(b)}(b,b2). Note that the hypotheses of Lemma \ref{lem:hypotheses_for_(a)(b)} do not require that $[y_{j_1}]\in [X]$, and here we have that $[y_{j_1}]\in [X]$ if and only $2\leq j_1\leq n-1$ by Definition \ref{def:Sn+-partialSn+-}. To show that $D^+([x^\delta y],\ii,\jj)=0$, we consider the  map
\[
\Omega\colon \D^+([x^\delta y],\ii,\jj)\to \D^+([x^\delta y],\ii,\jj)\text{ with }\Omega([x^{\delta'}y'])=[x^{1-\delta'}y'].
\]
To see that $\Omega$ is well defined, note that, if $[x^{\delta'}y']\in \D^+([x^\delta y],\ii,\jj)$, then $y'y_{j_1}\in Y^+$, $y'\in \partial Y^-$ by Definition \ref{def:Sn+-partialSn+-}, and $y'(\{j_1,j_1+1\})=\{1,2\}$ by Proposition \ref{prop:properties_of_T}(b). Thus, by Equations \eqref{equ:two_diagonal_entries_equal} and \eqref{equ:Yacts_diagonal_elements}, we obtain that the first two diagonal entries of $\ls{y'}e_i$ are equal for $i\neq j_1$. Thus, by Lemma \ref{lem:EXforPGU}(d), $x\in C_{\GU_n(q)}(\ls{y'}{E_{j_1}})=C_{\GU_n(q)}(\ls{y'}{\tau_\jj})$ and 
\[
[\ls{x^{1-\delta'}y'}\tau_\jj]=[\ls{x^{1-\delta'}}(\ls{x^{2\delta'-1}}(\ls{y'}\tau_\jj))]=[\ls{x^{\delta'}y'}\tau_\jj],
\]
so that $\Omega$ is well defined. Moreover, the map $\Omega$ is a non-trivial involution by Lemma \ref{lem:EXforPGU}(c), and thus we are done by Lemma \ref{lem:sum_set_with_involution_is_0} and Equation \eqref{equ:h_for_PGUn}. 
\end{proof}

\begin{remark}\label{rmk:why_x_and_Y-+}
The deduction that $y=1$ in the previous proof does not work if we allow $y$ to be any element of $Y$. For instance, according to the proof of Theorem \ref{thm:QDp_for_PGUn}, the following $(n-2)$-simplex is not stabilized by any non-trivial element of $Y^-$,
\[
\langle [e_{n-1}]\rangle<\langle [e_{n-1}],[e_{n-2}]\rangle<\cdots<\langle [e_{n-1}],\ldots,[e_2]\rangle<\langle [e_{n-1}],\ldots,[e_2],[e_1]\rangle,
\]
but, however, it is stabilized by $1\neq y_1\in Y\setminus Y^-$. The obstruction to consider all of $Y$ is that $Y$ normalizes $T$, so that $\Q\D_p$ cannot hold for $\langle T,Y\rangle$ by ``conical contractibility'' \cite[Proposition 2.4]{Quillen1978}. This explains why we keep just ``half'' of $Y$, $Y^-$, and add the quasi-reflection $x$ that does not normalize $T$; note that $\ls{x}e_1\notin T$ by Lemma \ref{lem:EXforPGU}(d) and Equation \eqref{equ:two_diagonal_entries_equal}. The proof of Theorem \ref{thm:QDp_for_PGUn} is based on that,  by the combinatorial properties in Proposition \ref{prop:properties_of_T}, $x$ fixes the ``boundary'' $\partial Y^-$ of $Y^-$, so that $Y^-$ and its ``reflection'' by $x$, $xY^-$, are correctly ``glued'' along their common boundary. We can think of $X=Y^-\cup xY^-$ as a version of $Y=Y^-\cup y_1Y^-$  ``twisted'' by $x$ instead of $y_1$.
\end{remark}

The rest of this section is devoted to the case of $\PGU_n(q)$ extended by field automorphisms of order $p$. If $q=s^{pl}$ for $p$ an odd prime, recall, from Definition \ref{def:Phi}, the field automorphism $\Phi$ of order $p$ of $\PGU_n(q)$ and that, by Proposition \ref{prop:pextensions_PSUnq}, we have $\rk_p(\PGU_n(q)\langle \Phi\rangle)=n$. We extend the bracket notation $[\cdot]$ for elements, subsets, functions, and simplices, introduced at the beginning of this section, to the quotient group $\PGU_n(q)\langle \Phi\rangle$ of $\GU_n(q)\langle \Phi\rangle$ by $Z$. We will need the following arithmetic digression.

\begin{lem}\label{lem:arithmethic_digression_for_PGUPhi}
Let $q=s^{pl}$ with $p$ odd, $p\mid q+1$, and $(p,q)\neq (3,8)$. Then there exists $\lambda\in \FF_{q^2}$ such that, for $\Lambda=\lambda^{1-s^{2l}}$, we have
\[
\begin{cases}
\lambda^{q+1}=1\text{, and}\\
\ord(\lambda)=\ord(\Lambda)=r\text{ for a prime $r\neq 2,3,5,p$.}
\end{cases}
\]
\end{lem}
\begin{proof}
By Fermat's little theorem, and because $p|(q+1)=(s^{pl}+1)$, we have that $p|(s^l+1)$. Now, by Zsigmondy's Theorem, there exists a prime $r$ that divides $s^{2pl}-1$ and does not divide $s^{kl}-1$ for $1\leq k< 2p$. This works but for $2p=1,2$ or $(p,q)=(3,8)$, and these cases are excluded by assumption. As $1\leq p<2p$, we get that $r$ does not divide $s^{pl}-1$ and, as $s^{2pl}-1=(s^{pl}-1)(s^{pl}+1)$, we deduce that $r$ divides $s^{pl}+1$. As $1\leq 2<2p$, $r$ does not divide $s^{2l}-1$ and, as $p\mid(s^l+1)$ and $(s^l+1)\mid (s^{2l}-1)$, we deduce that $r\neq p$. In addition, we have, by construction, that the order of $s^l$ modulo $r$ is $2p$. Then, by Fermat's little theorem again, we deduce that $2p\mid (r-1)$, so that $r\geq 1+2p\geq 7$ and $r\neq 2,3,5$. To finish, define $\lambda$ as any element of order $r$ of $\FF_{q^2}$.
\end{proof}

The next result serves as preparation for Theorem \ref{thm:QDp_for_PGUnfield}, and the only conclusion below that fails for $n=2$ is the third part of point (e). In fact, the case of $\PGU_2(q)$ extended by field automorphisms of order $p$ requires slightly different techniques, and we have omitted this case from the present work. In point (b) below, we recall that $(E_\Phi)_i$ is the subgroup of $E_\Phi$ spanned by the $e_k$ with $k\neq i$. 

\begin{lem}\label{lem:EXforPGUfield}
If $n\geq 3$, $q=s^{pl}$, $p$ is odd, $p\mid (q+1)$, and $(p,q)\neq (3,8)$, then there exist elements of order $p$, $e_1,e_2,e_3,\ldots,e_{n-1},e_n=\Phi$, in $\GU_n(q)\langle \Phi\rangle$, such that $e_i\in T$ for $1\leq i\leq n-1$ and that, for $E_\Phi=\langle e_1,e_2,\ldots,e_{n-1},e_n\rangle$, we have,
\begin{enumerate}
\item[(a)] $[E_\Phi]$ is an elementary abelian $p$-subgroup of $\PGU_n(q)\langle \Phi\rangle$ of rank $n$, and
\item[(b)] $y_i\in C_{\GU_n(q)}((E_\Phi)_i)$ for $1\leq i\leq n-1$.
\end{enumerate}
Moreover, there exists an element $x\in \GU_n(q)$ and a diagonal element $d\in \SU_n(q)$ such that, if we set $X=\{x^\delta y\mid \delta=0,1, y\in Y^-\}$, we have
\begin{enumerate}
\item[(c)] The map $\{0,1\}\times Y\to \PGU_n(q)$, $(\delta,y)\mapsto [x^\delta y]$, is injective,
\item[(d)] For any diagonal element $t=\diag(\alpha_1,\ldots,\alpha_n)\in T$, we have that $\ls{x}{t}$ is non-diagonal if $\alpha_1\neq \alpha_2$ and that $\ls{x}t=t$ if $\alpha_1=\alpha_2$.
\item[(e)] $[e_i,\Phi]=1$ for $i=1,\ldots,n-1$, $[x,\Phi]=1$, $[x,d]=1$, $[Y,\Phi]=1$, $\ls{(\ls{d}y)}t=\ls{y}t$ and $\ls{d}t=t$ for all $y\in Y$ and all $t\in T$, and $[\ls{d}\Phi]\notin [E_\Phi]$.
\end{enumerate}
\end{lem}
\begin{proof}
We obtain $e_1,\ldots,e_{n-1}$ and $x$ as follows: By Fermat's little theorem, and because $p|(q+1)=(s^{pl}+1)$, we get that $p|(s^l+1)$, and hence an element $\mu$ of order $p$ in $\FF_{s^{2l}}\leq \FF_{q^2}$ with $\mu\overline{\mu}=1$. Then we apply Lemma \ref{lem:EXforPGU} with this value of $\mu$, and with $q$ there equal to $s^l$ here, noting that $(p,q)\neq (3,8)$ implies that $s^l\geq 3$. Thus, we obtain $e_1,e_2,e_3,\ldots,e_{n-1}$, $E$, and $x$ satisfying the conclusions of Lemma \ref{lem:EXforPGU}, and, in addition, because $\FF_{s^{2l}}$ is the fixed field of $\lambda\mapsto\lambda^{s^{2l}}$, 
\begin{equation}\label{equ:e_ixPhi=1}
[e_i,\Phi]=1\text{ for $i=1,\ldots,n-1$ and }[x,\Phi]=1.
\end{equation}
If we set $E_\Phi=\langle e_1,\ldots,e_{n-1},e_n\rangle$, where $e_n=\Phi$, then it is clear that $[E_\Phi]$ is an elementary abelian $p$-subgroup of $\PGU_n(q)\langle \Phi\rangle$ of rank $n$, as $E_\Phi\cap Z=E\cap Z$ and $E\cap Z=1$ by Lemma \ref{lem:EXforPGU}(a), where we recall that $Z=Z(\GU_n(q))$. Thus, point (a) holds. Points (c) and (d) are consequence of points (c) and (d) in Lemma \ref{lem:EXforPGU}, respectively. The first and second parts of point (e) are exactly Equation \eqref{equ:e_ixPhi=1}, and the fourth part of point (e) follows from that $[y_i,\Phi]=1$ by Definition \ref{def:yiY}. Then the fourth part of point of (e) and Lemma \ref{lem:EXforPGU}(b) show that point (b) holds.

Now, from Lemma \ref{lem:arithmethic_digression_for_PGUPhi}, we obtain $\lambda$ and $\Lambda$ satisfying the properties there, and we define 
\begin{equation}\label{equ:element_diagonal_d}
d=\begin{cases}
\diag(\lambda,\lambda,\lambda^{-2})&n=3,\\
\diag(1,\ldots,1,\lambda,\lambda^{-1})&n\geq 4,
\end{cases}
\end{equation}
so that $d\in \SU_n(q)$. Then the third part of point (e) is consequence of Equation \eqref{equ:x_in_2x2_square}, and the fifth part of point (e) is a consequence of that $d$ is diagonal and of Equation \eqref{equ:Yacts_diagonal_elements}. Thus, we are left with proving the last part of point (e) and, to that aim, we perform the following computation,
\[
\ls{d}\Phi=d\Phi(d^{-1})\Phi=\begin{cases}
\diag(\Lambda,\Lambda,\Lambda^{-2})\Phi&n=3,\\
\diag(1,\ldots,1,\Lambda,\Lambda^{-1})\Phi&n\geq 4.
\end{cases}
\]
Assume now that $\ls{d}\Phi=ze\Phi^m$, where $0\leq m\leq p-1$, $e\in E$, and $z\in Z$. Then we immediately deduce that $m=1$ and, after raising to the $p$-th power, that $\Lambda^p=1$ if $n\geq 4$ or that $\Lambda^{3p}=1$ if $n=3$. But the order of $\Lambda$ is a prime not equal to $p$ or $3$ and we are done.
\end{proof}

Now we are ready to prove $\Q\D_p$ for $\PGU_n(q)$ extended by field automorphisms of order $p$.

\begin{thm}\label{thm:QDp_for_PGUnfield}
If $n\geq 3$, $p$ is odd, $p\mid (q+1)$, $(p,q)\neq (3,8)$, and $G$ is a $p$-extension of $\PGU_n(q)$ by field automorphisms of order $p$, then $\widetilde{H}_{n-1}(|\A_p(G)|;\ZZ)\neq 0$ and $\Q\D_p$ holds for $G$.
\end{thm}
\begin{proof}
We are assuming that $\PGU_n(q)$ admits a field automorphisms of odd order $p$, so that we may write $q=s^{pl}$. As two order-$p$ subgroups of field automorphisms of $\Aut(\PSU_n(q))$ are $\PGU_n(q)$-conjugate by Subsection \ref{subsection:preliminaries_unitary_groups}, we may assume without loss of generality that $G=\PGU_n(q)\langle \Phi\rangle$. 

Now, we apply Lemma \ref{lem:EXforPGUfield} and obtain $e_1,e_2,e_3,\ldots,e_{n-1},e_n$, $E_\Phi$, $x$, $d$, and $X$ satisfying the conclusions there. Next, consider the map $h\colon X\to \ZZ^*$ given by Equation \eqref{equ:h_for_PGUn} and the $(n-1)$-chain provided by Definition \ref{def:ZGXA} for the group $G$,
\[
\barsub=\barsub_{[E_\Phi],[X],[h]},
\]
as well as the chain obtained by conjugating $\barsub$ by $[d]$,
\[
\barsub'=\ls{[d]}\barsub=\barsub_{[\ls{d}E_\Phi],[\ls{d}X],[\ls{d}h]},
\]
where 
\[
\ls{d}E_\Phi=\ls{d}\langle e_1,e_2,e_3,\ldots,e_{n-1},e_n\rangle=\langle e_1,e_2,e_3,\ldots,e_{n-1},\ls{d}\Phi\rangle
\]
by the fifth part of Lemma \ref{lem:EXforPGUfield}(e) and because $e_i\in T$ for $1\leq i\leq n-1$,
\[
\ls{d}X=\ls{d}{\{x^\delta y\mid \delta=0,1, y\in Y^-\}}=\{x^\delta\ \ls{d}y\mid \delta=0,1, y\in Y^-\}
\]
by the third part of Lemma \ref{lem:EXforPGUfield}(e), and $\ls{d}h\colon \ls{d}X\to \ZZ^*$ is given by 
\begin{equation}\label{equ:h'_definition}
\ls{d}h(x^\delta\ \ls{d}y)=h(x^\delta y).
\end{equation}

We prove below that $\barsub-\barsub'$ is is a non-trivial cycle, showing that the conclusions of the statement are valid. For $\ii\in \tuples^n$, we use the notations $\sigma_\ii$, $\tau_\ii$, $\C(\cdot,\ii)$, $\D(\cdot,\ii,\jj)$, etc, of Section \ref{section:topological_preliminaries} and Equation \eqref{equ:D(x,i,j)} for $\barsub$, and their primed versions $\sigma'_\ii$, $\tau'_\ii$, $\D'(\cdot,\ii,\jj)$, etc, for $\barsub'$. Note that $[E_\Phi]$ is a maximal elementary abelian $p$-subgroup of $G$ by Lemma \ref{lem:EXforPGUfield}(a) and because $\rk_p(G)=n$ by Proposition \ref{prop:pextensions_PSUnq}. In particular, $\barsub-\barsub'$ cannot be a boundary.

\textbf{The $(n-1)$-chain $\barsub-\barsub'$ is non-trivial.} We show that, for the tuple $\ii=[n,1,2,\ldots,n-2]\in \tuples^{n}$, we have that
\[
\C([1],\ii)=\{([1],\ii)\},
\]
and that $[\sigma_\ii]$ is different from $[\ls{x^\delta \ (\ls{d}y)}\sigma'_\jj]$ for any $[x^\delta \ \ls{d}y] \in [\ls{d}X]$ and any $\jj\in \tuples^n$. In fact, write $\jj=[j_n,j_1,\ldots,j_{n-2}]$ with $\{j_1,\ldots,j_n\}=\{1,\ldots,n\}$, and assume that $[x^\delta \ \ls{d^\epsilon}y]$ with $\delta,\epsilon=0,1$, $y\in Y^-$, conjugates the $(n-1)$-simplex 
\[
\langle [\ls{d^\epsilon}e_{j_{n-1}}]\rangle<\langle [\ls{d^\epsilon}e_{j_{n-1}}],[\ls{d^\epsilon}e_{j_{n-2}}]\rangle<\cdots<\langle [\ls{d^\epsilon}e_{j_{n-1}}],[\ls{d^\epsilon}e_{j_{n-2}}],\ldots,[\ls{d^\epsilon}e_{j_1}]\rangle <[\ls{d^\epsilon}E_\Phi],
\]
into the $(n-1)$-simplex 
\[
[\sigma_\ii]=\langle [e_{n-1}]\rangle<\langle [e_{n-1}],[e_{n-2}]\rangle<\cdots<\langle [e_{n-1}],[e_{n-2}],\ldots,[e_1]\rangle<[E_\Phi].
\] 
Then we want to show that $\delta=\epsilon=0$, $y=1$, and $\jj=\ii$.

Note that, in the simplex $[\sigma_\ii]$, $\Phi$ only belongs to the right-most subgroup, so that we must have $\{j_{n-1},\ldots,j_1\}=\{1,\ldots,n-1\}$. In addition, we have that
\[
\ls{x^\delta \ (\ls{d^\epsilon}y)} \ {\ls{d^\epsilon}{e_{j_k}}}=\ls{x^\delta \ y}e_{j_k}
\]
for $k=1,\ldots,n-1$, and this follows immediately if $\epsilon=0$, or by the fifth part of Lemma \ref{lem:EXforPGUfield}(e) if $\epsilon=1$. Then the arguments in the proof of Theorem \ref{thm:QDp_for_PGUn} show that $\delta=0$, $y=1$, and $\jj=\ii$. To finish, examination of the right-most subgroups of the $(n-1)$-simplices above shows that 
\[
[\ls{d^\epsilon}\Phi]\in [E_\Phi],
\]
and, by the last part of Lemma \ref{lem:EXforPGUfield}(e), we must have $\epsilon=0$.

\textbf{The $(n-1)$-chain $\barsub-\barsub'$ is a cycle.} Consider tuples $\ii,\jj\in \tuples^n$, where $\jj=[j_1,\ldots,j_{n-1}]\in \tuples^n$, and an element $x^\delta y\in X$. 

If $j_1=1,\ldots,n-1$, then the arguments in the proof of Theorem \ref{thm:QDp_for_PGUn} show that $D([x^\delta y],\ii,\jj)=0$, where we use Lemma \ref{lem:EXforPGUfield}(b) instead of Lemma \ref{lem:EXforPGU}(b) and that, by the second part of Lemma \ref{lem:EXforPGUfield}(e), $[x,\Phi]=1$. Moreover, as $\D'([x^\delta \ \ls{d}y],\ii,\jj)$ is the $[d]$-conjugate of $\D([x^\delta y],\ii,\jj)$, we deduce that $D'([x^\delta \ \ls{d}y],\ii,\jj)=0$ too.

If $j_1=n$, we show that $D([x^\delta y],\ii,\jj)+D'([x^\delta \ \ls{d} y],\ii,\jj)=0$ by means of the map
\[
\Omega\colon \D([x^\delta y],\ii,\jj)\to \D'([x^\delta \ \ls{d}y],\ii,\jj)\text{ with }\Omega([x^{\delta'}\ y'])=[x^{\delta'} \ \ls{d}{y'}].
\]
To check that $\Omega$ is well defined, note first that, as $j_1=n$, the subgroups appearing in the chain $\tau_\jj$ are spanned by diagonal elements, i.e., they do not contain $e_n=\Phi$. As $ [\ls{x^{\delta'}y'}\tau_\jj]=[\ls{x^\delta y}\tau_\ii]$, then the subgroups in $\tau_\ii$ are also spanned by diagonal elements. Thus, by the fifth part of Lemma \ref{lem:EXforPGUfield}(e), we have $ \tau'_\jj=\tau_\jj$,  $\tau'_\ii=\tau_\ii$,
\[
\ls{(\ls{d}y')}\tau'_\jj=\ls{(\ls{d}y')}\tau_\jj=\ls{y'}\tau_\jj\text{, }\ls{(\ls{d}y)}\tau'_\ii=\ls{(\ls{d}y)}\tau_\ii=\ls{y}\tau_\ii,
\]
$\ls{x^{\delta'} \ (\ls{d}y')}{\tau'_\jj}=\ls{x^{\delta'} y'}{\tau_\jj}$, and $\ls{x^{\delta} \ (\ls{d}y)}{\tau'_\ii}=\ls{x^{\delta} y}{\tau_\ii}$. Thus, $\Omega$ is well defined. That $\Omega$ is a bijection follows easily from Lemma \ref{lem:EXforPGUfield}(c) and the third part of Lemma \ref{lem:EXforPGUfield}(e). To finish, note that Equation \eqref{equ:h'_definition} gives that
\[
[h]([x^{\delta'}y'])-[\ls{d}h]([x^{\delta'} \ \ls{d}{y'}]) = h(x^{\delta'}y')-\ls{d}h(x^{\delta'} \ \ls{d}{y'})=h(x^{\delta'}y')-h(x^{\delta'}{y'})=0.
\]
\end{proof}

\begin{remark}\label{rmk:some_nonidentified_simplices}
In the proofs of Theorem \ref{thm:QDp_for_PGUn} and \ref{thm:QDp_for_PGUnfield}, we have shown that the simplex 
\[
\langle [e_{n-1}]\rangle<\langle [e_{n-1}],[e_{n-2}]\rangle<\cdots<\langle [e_{n-1}],\ldots,[e_2]\rangle<\langle [e_{n-1}],\ldots,[e_2],[e_1]\rangle=[E],
\] 
or the simplex
\[
\langle [e_{n-1}]\rangle<\langle [e_{n-1}],[e_{n-2}]\rangle<\cdots<\langle [e_{n-1}],\ldots,[e_2]\rangle<\langle [e_{n-1}],\ldots,[e_2],[e_1]\rangle< [E_\Phi],
\]
respectively, is not identified with any other simplex of the chains constructed there. Nevertheless, some identifications take place as, for instance, for $n\geq 3$, $y_{n-1}\in Y^-$ stabilizes the $(n-2)$-simplex
\[
\langle [e_1]\rangle<\langle [e_1],[e_2]\rangle<\cdots<\langle [e_1],\ldots,[e_{n-2}]\rangle<\langle [e_1],\ldots,[e_{n-2}],[e_{n-1}]\rangle.
\]
\end{remark}

\begin{exa}\label{example:n2_quasireflections}
The chain constructed for $\PGU_3(q)$ is depicted in Figure \ref{fig:triangulationPGU3PGU3Phi} (left),  with bullet points and squares corresponding, respectively, to $p$-rank 1 and 2 elementary abelian $p$-subgroups. The $1$-simplex $\langle [e_2]\rangle<\langle [e_2],[e_1]\rangle$, which is not identified with any other $1$-simplex, is marked in white, and the following two identified $1$-simplices, see Remark \ref{rmk:some_nonidentified_simplices}, are marked in black,
\[
\langle [e_1]\rangle<\langle [e_1],[e_2]\rangle=\ls{y_2}(\langle [e_1]\rangle<\langle [e_1],[e_2]\rangle).
\]
Note that the left-lower half is similar to that in Figure \ref{fig:triangulationS7r=2}, see also Example \ref{exa:Sn+-partialSn+-}, and that the right-upper half is obtained from the left-lower half by conjugation by the quasi-reflection $x$. The elements
\[
e_2=\diag(\mu,\mu,1)\text{ and }\ls{y_2y_1}e_1=\diag(1,1,\mu)
\]
are fixed by conjugation by $x$, see Lemma \ref{lem:EXforPGU}(d). In addition, the elements $e_1$ and $\ls{x}e_1$ are fixed by conjugation by $y_2\in X$ ($e_1\mapsto \ls{y_2}e_1=e_1$ and $\ls{x}e_1\mapsto \ls{xy_2}e_1=\ls{x}e_1$), and the elements $\ls{y_2}e_2$ and $\ls{xy_2}e_2$ are fixed by conjugation by $y_1\not\in X$.

The chain constructed for $\PGU_3(q)$ extended by field automorphisms of order $p$ is depicted in Figure \ref{fig:triangulationPGU3PGU3Phi} (right), and labels are included for some of the conjugating elements of  $X$ and some of the $p$-rank 1 elementary abelian $p$-subgroups. Moreover, the horizontal equator coincides with the circle on the left. Besides the reflection induced by conjugation by $x$, now there exists another symmetry induced by conjugation by the diagonal element $d$, and recall that $[x,d]=1$ by the third part of Lemma \ref{lem:EXforPGUfield}(e). The elements $e_2$ and $\ls{y_2y_1}e_1$ are fixed by both symmetries and the elements at the apexes, $e_3=\Phi$ and $\ls{d}e_3$, are fixed by conjugation by $x$. In addition, some barycentric subdivisions are shown, see also Figure \ref{fig:barycentric_subdivision}. The white $2$-simplex corresponds to 
\[
\langle [e_2]\rangle<\langle [e_2],[e_1]\rangle<\langle [e_2],[e_1],[e_3]\rangle,
\]
and it is not identified with any other $2$-simplex. The two black $2$-simplices are equal.
\end{exa}

\begin{figure}[h!]
\centering

    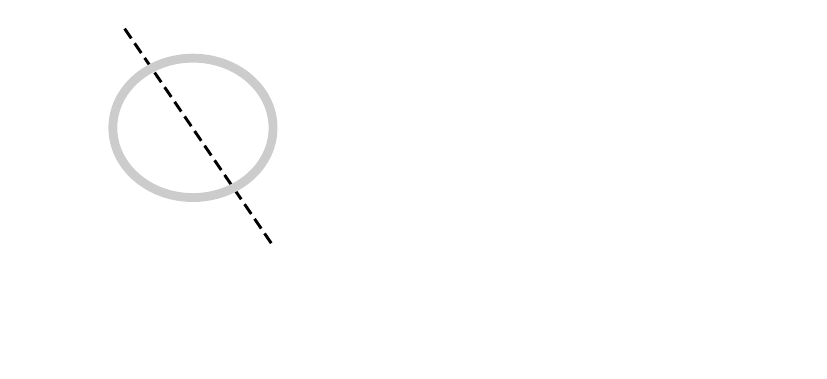

\caption{Complex for $\PGU_3(q)$ (left) and $\PGU_3(q)\langle \Phi\rangle$ (right).}
\label{fig:triangulationPGU3PGU3Phi}
\end{figure}

\section{Proof of Theorem B.}
\label{section:TheoremB}

We dedicate this section to provide details for the proof of Theorem B stated in the introduction, and we employ \cite[Proposition 4.8]{AS1993}.

\begin{proof}[Proof of Theorem B]
As preparation, we refer the reader to the account, in Proposition \ref{prop:pextensions_PSUnq}, of the $p$-extensions $LB$ of $L=\PSU_n(q)$ and their $p$-ranks, as well as the $p$-ranks of $\PGU_n(q)$ and its extension by field automorphisms of order $p$.

As additional setup, we make a short detour to reformulate the conclusions of Theorems \ref{thm:QDp_for_PGUn} and \ref{thm:QDp_for_PGUnfield} in some cases: An a posteriori analysis of the cycles constructed in these two theorems, reveals that the elementary abelian $p$-subgroups appearing in the simplices involved in those cycles, have generators of the form
\begin{equation}\label{equ:ThmB_generators}
\text{ either }\ls{[g]}[e]\text{ or }\Phi\text{ or }\ls{[d]}\Phi,
\end{equation}
where $g\in \GU_n(q)$, $e$ is a diagonal element of $T\leq \GU_n(q)$ of order $p$ by Lemmas \ref{lem:EXforPGU} and \ref{lem:EXforPGUfield}, and $d\in \SU_n(q)$ by Equation \eqref{equ:element_diagonal_d}. The second and third cases occur only in Theorem \ref{thm:QDp_for_PGUnfield}, and we have employed that $g\in \langle x,Y,d\rangle$ with $[x,\Phi]=[Y,\Phi]=1$ by Lemma \ref{lem:EXforPGUfield}(e). Moreover, we clearly have
\begin{equation}\label{equ:ThmB_23generators_in_PSUnq}
\text{$\Phi$ and $\ls{[d]}\Phi$ belong to $\PSU_n(q)\langle \Phi\rangle$.}
\end{equation}
We show next that, if $n_p<(q+1)_p$, then we also have
\begin{equation}\label{equ:ThmB_1generator_in_PSUnq}
\ls{[g]}[e]\in \PSU_n(q),
\end{equation}
where $\ls{[g]}[e]$ is the first generator listed in Equation \eqref{equ:ThmB_generators}. In fact, as $[g]$ normalizes $\PSU_n(q)$ by Equation \eqref{equ:PSU-PGU-cyclic}, it is enough to show that $[e]\in \PSU_n(q)$. This is a well known fact, see for instance the proof of Lemma \cite[Lemma 4.7(a)]{AS1993}, and we provide details: As $n_p<(q+1)_p$, there exists an element $\omega$ in $\FF^*_{q^2}$ of order $\ord(\omega)=pn_p$. In particular, $w^{q+1}=1$ and we may assume that $e=\diag(\mu^{i_1},\ldots,\mu^{i_n})$ with $\mu=\omega^{n_p}$. Then, if we write $n=n_pn'$ with $(n',p)=1$, we have, for any $m\in \ZZ$, that
\[
[e]=[e\diag(\omega^m,\ldots,\omega^m)]\text{ and }\det(e\diag(\omega^m,\ldots,\omega^m))=\omega^{n_p(i+n'm)},
\]
where $i=i_1+\ldots+i_n$. Thus, to obtain that $[e]\in \PSU_n(q)$, it is enough to find $m$ with $i+n'm\equiv 0\pmod{p}$, and this is possible because $n'$ is coprime to $p$. We conclude from Equations \eqref{equ:ThmB_generators}, \eqref{equ:ThmB_23generators_in_PSUnq}, and \eqref{equ:ThmB_1generator_in_PSUnq}, that, if $n_p<(q+1)_p$, then
\begin{equation}\label{equ:ThmB_PSUnq_n_p<(q+1)_p}
\widetilde{H}_{n-2}(|\A_p(\PSU_n(q))|;\ZZ)\neq 0\text{ and }
\widetilde{H}_{n-1}(|\A_p(\PSU_n(q)\langle \Phi\rangle|;\ZZ)\neq 0,
\end{equation}
under the assumptions of Theorems \ref{thm:QDp_for_PGUn} or \ref{thm:QDp_for_PGUnfield}, respectively, and where we have employed that $\PGU_n(q)$ and $\PSU_n(q)$, as well as its extensions by $\langle \Phi\rangle$, have the same $p$-rank if $n_p<(q+1)_p$ by Proposition \ref{prop:pextensions_PSUnq}.

We now come back to the statement of Theorem B and treat first the low dimensional cases $n=2,3$: For the case $n=2$, we have $\PSU_2(q)\cong\PSL_2(q)$, and $p$-extensions of $\PSL_2(q)$ satisfy $\Q\D_p$ by \cite[Theorem 3.1]{AS1993}, but for the excluded cases $(p,q)=(3,2)$, and $(p,q)=(3,8)$ if $B$ contains field automorphisms. If $n=3$ and $B$ does not contain fields automorphisms, then $\PSU_3(q)B$ satisfies $\Q\D_p$ by \cite[Proposition 4.8]{AS1993} but for $3\geq q(q-1)$, i.e., for the excluded case $(p,q)=(3,2)$. If $n=3$ and $B$ contains field automorphisms, then we have $n_p\leq 3\leq p$, and $p^2\leq (q+1)_p$ by \cite{Artin1955}, so that $n_p<(q+1)_p$. Thus, in the rest of the proof, we may assume that
\begin{equation}\label{equ:ThmB_assumption_n_q+1}
\text{either }n\geq 3\text{ and }n_p<(q+1)_p,\text{ or }n\geq 4\text{ and }n_p\geq (q+1)_p.
\end{equation}

For conciseness, we explicitly consider the possible generators of $B$: a diagonal automorphism $[g_1]$ of order $p$, and a field automorphism $\ls{[g_2]}\Phi^i$, where $i$ is co-prime to $p$, and $g_1,g_2\in \GU_n(q)$, see Subsection \ref{subsection:automorphisms_unitary:groups}. After raising to the appropriate power and conjugating $\PSU_n(q)$ by $[g_2^{-1}]$, we may assume that,
\begin{equation}\label{equ:can_use_Phi}
\text{If $B$ contains field automorphisms of order $p$, then $\Phi\in B$.}
\end{equation}
Now we proceed examining the $p$-extensions $LB$ of $L=\PSU_n(q)$ one by one.

\textbf{(i) $B=1$:} Then $LB=L$ and, if $n\geq 3$ and $n_p<(q+1)_p$, we are done by Equation \eqref{equ:ThmB_PSUnq_n_p<(q+1)_p}, where we exclude the case $(p,q)=(3,2)$. Otherwise, by Equation \eqref{equ:ThmB_assumption_n_q+1}, we have $n\geq 4$ and $n_p\geq (q+1)_p$, so that $p\mid n$ and, by \cite[Lemma 4.7(b)]{AS1993}, 
\[
\A_p(\PSU_{n-1}(q))\subseteq \A_p(\PSU_n(q)),
\]
where the groups involved have $p$-rank $n-2$. As $n-1\geq 3$ and $(n-1)_p=1$, we can apply the case above to the left-hand side.

\textbf{(ii) $B$ is cyclic generated by a field automorphism of order $p$:} If $n\geq 3$ and $n_p<(q+1)_p$, we are done by Equation \eqref{equ:ThmB_PSUnq_n_p<(q+1)_p}, where we are excluding the case $(p,q)=(3,8)$. Otherwise, $n\geq 4$ and $n_p\geq (q+1)_p$, and we use \cite[Lemma 4.7(b)]{AS1993} upon the addition of field automorphisms to obtain
\[
\A_p(\PSU_{n-1}(q)B)\subseteq \A_p(\PSU_n(q)B),
\]
where the groups involved have $p$-rank $n-1$. Note that $B$ acts on $\PSU_{n-1}(q)$ by  Equation \eqref{equ:can_use_Phi}. As $n-1\geq 3$ and $(n-1)_p=1$, we may apply here the previous case $n_p<(q+1)_p$ of this point (ii) to the left-hand side.

\textbf{(iii) $B$ is cyclic generated by a diagonal automorphisms of order $p$:} Equation \eqref{equ:PSU-PGU-cyclic} gives $\A_p(\PSU_n(q)B)=\A_p(\PGU_n(q))$, and this case follows from Theorem \ref{thm:QDp_for_PGUn} and the fact that $\rk_p(\PSU_n(q)B)=n-1$ by Proposition \ref{prop:pextensions_PSUnq}.

\textbf{(iv) $B$ has $p$-rank $2$ and it is generated by a field and a diagonal automorphism of order $p$:} By Equation \eqref{equ:can_use_Phi}, we have $B=\langle [g_1],\Phi\rangle$ and, by Equations \eqref{equ:ThmB_generators} and \eqref{equ:ThmB_23generators_in_PSUnq}, to show that the cycle constructed in Theorem \ref{thm:QDp_for_PGUnfield} belongs to $|\A_p(\PSU_n(q)B)|$, it is enough to show that the generators of the form $\ls{[g]}{[e]}$ belong to $\PSU_n(q)\langle [g_1]\rangle$, where $e\in T\leq \GU_n(q)$ has order $p$ and $g\in \GU_n(q)$. But $\ls{[g]}{[e]}=[\ls{g}e]\in \PGU_n(q)$ has order $p$ and, by Equation \eqref{equ:PSU-PGU-cyclic} again, we have, as in the earlier case (iii), that $\A_p(\PSU_n(q)\langle [g_1]\rangle)=\A_p(\PGU_n(q))$, so that the desired conclusion follows. Finally, as $\rk_p(\PSU_n(q)B)=n$ by Proposition \ref{prop:pextensions_PSUnq}, we conclude that $\PSU_n(q)B$ has $\Q\D_p$.
\end{proof}

\begin{remark}
Note that all non-trivial cycles constructed in Theorems \ref{thm:QDp_for_PGUn}, \ref{thm:QDp_for_PGUnfield}, and B, have an explicit description, a fact that might have further consequences, see \cite[p.12]{P2021} for instance.
\end{remark}


\begin{thebibliography}{99}

\bibitem{Artin1955} E.Artin, {\it The orders of the linear groups}, Comm. Pure Appl. Math., 8 (1955), 355--365.

\bibitem{AKO2011} M. Aschbacher, R. Kessar, B. Oliver, {\it Fusion systems in algebra and topology}, London Mathematical Society Lecture Note Series, vol. 391, Cambridge University Press, Cambridge, 2011.
 
\bibitem{AK1990} M.~Aschbacher, P.B.~Kleidman, {\it On a conjecture of Quillen and a lemma of Robinson}, Arch. Math. (Basel), 55 (1990), no.3, 209--217.  

\bibitem{AS1993} M.~Aschbacher, S.~Smith, {\it On Quillen's conjecture for the p-groups complex}, Ann. of Math. (2) 137 (1993), no. 3, 473--529.  
  
\bibitem{BS2008} D.J. Benson, S.D. Smith, {\it Classifying spaces of sporadic groups}, Mathematical Surveys and Monographs, vol. 147, American Mathematical Society, Providence, RI, 2008.  
  
\bibitem{Brown75} K.S. Brown, {\it Euler characteristics of groups: the {$p$}-fractional part}, Invent. Math., 29, (1975), 1, 1--5.

\bibitem{Diaz2016} A. D\'iaz Ramos, {\it On Quillen's conjecture for p-solvable groups}, Journal of Algebra, 513, 1, (2018), 246--264.

\bibitem{DiazMazza2020} A. D\'iaz Ramos, N. Mazza, {\it A geometric approach to Quillen's conjecture}, J. Group Theory 25 (2022), 91--112.

\bibitem{Dwyer1997} W. G. Dwyer, {\it Homology decompositions for classifying spaces of finite groups}, Topology 36, (1997), 783--804.

\bibitem{GL} D. Gorenstein, R. Lyons, {\it The local structure of finite groups of characteristic 2 type}. Mem. Amer. Math. Soc. 42 (1983), no. 276. 

\bibitem{GLSI} D.~Gorenstein, R.~Lyons, R.~Solomon, {\it The classification of the finite simple groups}.  Mathematical Surveys and Monographs, Volume 40.1, AMS, 1994.
  
\bibitem{GLSIII} D.~Gorenstein, R.~Lyons, R.~Solomon, {\it The classification of the finite simple groups}. Number 3. Mathematical Surveys and Monographs, Volume 40.3, AMS, 1998.

\bibitem{Grove} L.C. Grove, {\it Classical Groups and Geometric Algebra}, Graduate Studies in Mathematics, Volume 39, AMS, 2002.

\bibitem{HawksIsaacs1988} T. Hawkes, I. M. Isaacs, {\it On the poset of $p$-subgroups of a $p$-solvable group},
J. Lond. Math. Soc. (2) 38 (1988), no. 1, 77--86.

\bibitem{KasselTuraev} C.~Kassel, V.~Turaev, {\it Braid Groups}, Graduate Text in Mathematics, volume 247, Springer, 2008.

\bibitem{P2021} K.I. Piterman, {\it An approach to Quillen's conjecture via centralisers of               simple groups}, Forum Math. Sigma, vol. 9, (2021), Paper No. e48, 23.

\bibitem{PSC2021} K.I. Piterman, I. Sadofschi Costa, A. Viruel, {\it Acyclic 2-dimensional complexes
and Quillen's conjecture}, Publ. Mat. 65 (2021), no. 1, 129--140. 

\bibitem{PS2022} K. I. Piterman, S.D. Smith, {\it Some results on Quillen's Conjecture via equivalent-poset techniques}, J. Comb. Algebra 9 (2025), no. 3/4, 265–387.

\bibitem{Quillen1978} D.~Quillen, {\it Homotopy properties of the poset of nontrivial p-subgroups of a group},  Adv. Math. 28 (1978), no. 2,   101--128. 

\bibitem{Robinson88} G.R. Robinson, {\it Some remarks on permutation modules}, J. Algebra, vol. 118, (1988), no. 1, 46--62.

\bibitem{Smith2011} S.D. Smith, {\it Subgroup complexes}, Mathematical Surveys and Monographs, vol. 179, American Mathematical Society, Providence, RI, 2011.

\bibitem{Webb1987} P.J. Webb, {\it Subgroup complexes}, The Arcata Conference on Representations of Finite Groups (Arcata, Calif., 1986), Proc. Sympos. Pure Math., vol. 47, Amer. Math. Soc., Providence, RI, 1987, 349--365.

\end{thebibliography}
\end{document}